\documentclass[11pt,reqno]{amsart}
\usepackage[T1]{fontenc}
\usepackage[utf8]{inputenc}
\usepackage{newtxtext,newtxmath}
\usepackage[a4paper,textwidth=15.4cm,textheight=23.2cm,centering]{geometry}
\usepackage{microtype,mathtools,enumitem,booktabs}
\usepackage[hidelinks,unicode]{hyperref}
\usepackage[nameinlink,noabbrev]{cleveref}
\allowdisplaybreaks[2]
\newtheorem{theorem}{Theorem}[section]
\newtheorem{proposition}[theorem]{Proposition}
\newtheorem{lemma}[theorem]{Lemma}
\newtheorem{corollary}[theorem]{Corollary}
\theoremstyle{definition}

\theoremstyle{remark}

\numberwithin{equation}{section}
\newcommand{\CP}{\mathbb{CP}}
\newcommand{\R}{\mathbb R}
\newcommand{\C}{\mathbb C}
\newcommand{\Z}{\mathbb Z}
\newcommand{\g}{\mathfrak g}
\newcommand{\s}{\mathfrak s}
\newcommand{\z}{\mathfrak z}
\newcommand{\kk}{\mathfrak k}
\newcommand{\LL}{\mathcal L}
\newcommand{\FF}{\mathcal F}
\newcommand{\OO}{\mathcal O}
\newcommand{\HH}{\mathcal H}
\newcommand{\DD}{\mathcal D}
\newcommand{\Ric}{\operatorname{Ric}}
\newcommand{\Hess}{\operatorname{Hess}}
\newcommand{\Vol}{\operatorname{Vol}}
\newcommand{\tr}{\operatorname{tr}}
\newcommand{\rank}{\operatorname{rank}}

\newcommand{\Ad}{\operatorname{Ad}}

\newcommand{\diver}{\operatorname{div}}
\newcommand{\dd}{\,d}
\newcommand{\dbar}{\bar\partial}
\newcommand{\Lie}{\mathcal L}
\newcommand{\RCD}{\mathrm{RCD}}
\newcommand{\GH}{\mathrm{GH}}

\newcommand{\HS}{\mathrm{HS}}
\newcommand{\op}{\mathrm{op}}
\newcommand{\FS}{\mathrm{FS}}
\newcommand{\KKS}{\mathrm{KKS}}
\newcommand{\norm}[1]{\left\|#1\right\|}
\newcommand{\ip}[2]{\langle #1,#2\rangle}
\newcommand{\Arxiv}[1]{\href{https://arxiv.org/abs/#1}{arXiv:#1}}
\newcommand{\DOI}[1]{\href{https://doi.org/#1}{doi:\nolinkurl{#1}}}
\title[Spectral almost rigidity]{Spectral almost rigidity  on K\"ahler manifolds\break with positive Ricci lower bound}
\author{Haohao Wang}
\date{}
\hypersetup{pdftitle={Sharp spectral almost rigidity on K\"ahler manifolds with positive Ricci lower bound},pdfsubject={The Petersen--Aubry K\"ahler analogue, sharp noncollapsing, and symmetry loss},pdfkeywords={Kahler geometry, spectral almost rigidity, Petersen, Aubry, collapse, symmetry loss},pdflang={en-US}}
\begin{document}
\begin{abstract}
In this work, a sharp K\"ahler  spectral almost-rigidity theorem was established, resolving Conjecture~1.8 of Chu--Wang--Zhang. For compact K\"ahler manifolds satisfying $\Ric(\omega)\geq\omega$, pinching the first $n^2+3$ nonzero complex eigenvalues to one forces  the manifolds to be Gromov–Hausdorff close  to normalized complex projective space and determines the biholomorphism type.  The smaller index $n^2+1$ is the sharp threshold for noncollapsing. More generally, if a normalized measured limit has essential real dimension $r$, then the multiplicity of the critical eigenvalue is at most $r+\lfloor r/2\rfloor^2$, with equality attained in every dimension. A uniform energy estimate for kernel projections of complex gradient Gram matrices leads to a Lie algebra action on the whole spectral resolution. 
\end{abstract}
\maketitle
\section{Introduction}

\subsection{Eigenvalue comparison and almost rigidity}
Throughout the paper, manifolds are compact, connected, and without boundary. Write $\nu_k$ for the $k$th nonzero eigenvalue of the real Laplacian. On a Riemannian $m$-manifold, $m\geq2$, with $\Ric_g\geq(m-1)g$, the Lichnerowicz estimate gives $\nu_1\geq m$ \cite{Lichnerowicz}, and Obata's theorem characterizes the equality case as the unit sphere \cite{Obata}.

Petersen's spectral pinching theorem \cite{Petersen}, sharpened by Aubry \cite[Theorem~1]{Aubry}, gives the corresponding almost-rigidity statement. For every $\epsilon>0$ there is a $\delta(m,\epsilon)>0$ such that
\[
 \Ric_g\geq(m-1)g,\qquad \nu_m\leq m+\delta(m,\epsilon)
\]
imply that $M$ is diffeomorphic to $S^m$ and is $\epsilon$-close to the unit sphere in Gromov--Hausdorff distance. Petersen's original condition involved $m+1$ eigenvalues; Aubry obtained the optimal index $m$, whereas control of only $m-1$ eigenvalues does not imply metric closeness to the sphere \cite[\S7]{Aubry}. For the degree issue in the original spectral-map argument, see Petersen's erratum \cite{PetersenErratum} and Aubry's treatment in \cite[\S3]{Aubry}. No a priori volume lower bound is required in the sphere theorem.

Chu--Wang--Zhang explicitly proposed a K\"ahler analogue of this theorem \cite[\S1.3]{CWZ}. On a K\"ahler manifold $(M^n,g,J,\omega)$, we use
\begin{equation}\label{eq:conventions}
 \omega(V,W)=g(JV,W),\qquad
 \Delta f=g^{i\bar j}f_{i\bar j},\qquad \Delta_{\R}=2\Delta.
\end{equation}
The positive eigenvalues of $-\Delta$, repeated according to their real multiplicities, are denoted by $\lambda_1\leq\lambda_2\leq\cdots$; thus the real eigenvalues are $\nu_k=2\lambda_k$. The normalized Fubini--Study metric on $\CP^n$ is denoted by $\omega_n$, so that $\Ric(\omega_n)=\omega_n$. Its first complex eigenvalue is one and has multiplicity $n^2+2n$ \cite[Lemma~6.1]{CWZ}.

The K\"ahler eigenvalue comparison theorem gives $\lambda_1\geq1$ under $\Ric(\omega)\geq\omega$ \cite[Theorem~1.5]{CWZ}. In contrast with the Riemannian equality case, $\lambda_1=1$ does not characterize the comparison model. A K\"ahler--Einstein product can have many eigenfunctions with eigenvalue one without being a complex space form. Multiplicity is therefore essential, and the K\"ahler problem is not a direct specialization of the sphere theorem.

Chu--Wang--Zhang proved that $\lambda_{n^2+3}=1$ characterizes $(\CP^n,\omega_n)$ among compact K\"ahler manifolds with $\Ric(\omega)\geq\omega$ \cite[Theorem~1.6]{CWZ}. Their result is sharp, since the normalized product $\CP^{n-1}\times\CP^1$ has $n^2+2$ eigenfunctions with eigenvalue one. In Conjecture~1.8 of \cite{CWZ}, they asked for the corresponding almost rigidity under the Ricci lower bound alone. They established it for K\"ahler--Einstein metrics and, in the almost K\"ahler--Einstein setting, for metrics in the anticanonical class \cite[Theorems~1.9--1.10]{CWZ}. Without these assumptions, they identified possible volume collapse as an additional obstacle.

Our first result proves the conjecture under its original Ricci and spectral assumptions.
\begin{theorem}\label{thm:mainrigidity}
Let $n\geq2$. For every $\epsilon>0$ there exists $\delta(n,\epsilon)>0$ such that, if
\[
 \Ric(\omega)\geq\omega,\qquad
 \lambda_{n^2+3}(M,\omega)\leq1+\delta(n,\epsilon),
\]
then $M$ is biholomorphic to $\CP^n$ and
\[
 d_{\GH}\bigl((M,g),(\CP^n,g_n)\bigr)<\epsilon.
\]
\end{theorem}
No volume lower bound or assumption on $[\omega]$ is imposed. \Cref{thm:mainrigidity} resolves \cite[Conjecture~1.8]{CWZ} and gives the sharp K\"ahler counterpart of the Petersen--Aubry sphere theorem. The optimal pinching index is $n^2+3$, rather than the real dimension: the product above excludes $n^2+2$. The result is qualitative; no explicit power-law distance estimate, as in \cite[Theorem~1]{Aubry}, is asserted.

There is a different, smaller threshold for noncollapsing. We use Riemannian volume, so that $\Vol_g(M)=\int_M\omega^n/n!$.
\begin{theorem}\label{thm:noncollapseintro}
For each $n\geq2$, there are $\delta_0(n)>0$ and $v_0(n)>0$ such that, if
\[
 \Ric(\omega)\geq\omega,\quad \lambda_{n^2+1}(M,\omega)\leq1+\delta_0(n),
\]
then $\Vol_g(M)\geq v_0(n).$

The index $n^2+1$ is optimal. There are smooth K\"ahler metrics on $\CP^{n-1}\times\CP^1$ with $\Ric(\omega_j)\geq\omega_j$, $\lambda_{n^2}(\omega_j)\to1$, and $\Vol_{g_j}(M)\to0$.
\end{theorem}

The thresholds in \cref{thm:mainrigidity,thm:noncollapseintro} distinguish exclusion of collapse from identification of the projective-space model. As in the Riemannian sphere theorem, noncollapsing is a consequence rather than an assumption; the additional point here is the sharp lower index $n^2+1$. The proof separates noncollapsing from model identification. It first constructs a Lie algebra from the whole critical spectral cluster, including the cluster directions which are not specified in the pinching assumption. A kernel-projection estimate then rules out a varying nullspace for commuting spectral gradients. At high multiplicity, this eliminates the center and gives full rank. The limiting action is constructed on every eigenspace, and not merely on the critical one. Its coadjoint image provides both noncollapsing and transitivity. Only after these steps do we recover a smooth K\"ahler structure and use the exact rigidity theorem of \cite{CWZ}.

\subsection{Dimension and disappearing symmetries}
For a sequence as above, set
\[
 \mu_j=\frac{\omega_j^n}{\int_{M_j}\omega_j^n}.
\]
After passing to a subsequence, the metric measure spaces converge to a compact $\RCD(1,2n)$ probability space $(Z,d,\mu)$. We write $\LL=-\Delta_\mu$ for its nonnegative real Laplacian and
\[
 E=\ker(\LL-2),\qquad q=\dim_{\R}E.
\]
The essential dimension $r$ is the almost-everywhere dimension of the tangent module. This dimension is constant by \cite{BS}. For a point, we set $r=0$ and $q=0$. We do not identify essential and Hausdorff dimensions in general.

\begin{theorem}\label{thm:dimensionintro}
Every such limit satisfies
\begin{equation}\label{eq:dimensionintro}
 q\leq Q(r):=r+\left\lfloor\frac r2\right\rfloor^2.
\end{equation}
For every fixed original complex dimension $n$ and every $0\leq r\leq2n$, equality is attained by a sequence of smooth K\"ahler manifolds with $\Ric(\omega_j)\geq\omega_j$. The equality models can be chosen as
\[
 Z_{2k}=\CP^k,\qquad Z_{2k+1}=\CP^k\times I,
 \qquad I=\left([0,\pi],dt^2,\tfrac12\sin t\,dt\right),
\]
where the projective factors have their normalized metrics and probability measures.
\end{theorem}

In particular, retaining $p$ critical modes forces
\begin{equation}\label{eq:inverseprofile}
 r\geq 2\lfloor\sqrt p\rfloor-\mathbf1_{\{p\text{ is a square}\}}.
\end{equation}
Thus the dimension lost under collapse is quantitatively constrained by the number of surviving modes. All these dimension bounds are sharp in the sense of retaining at least $p$ modes.

The critical space carries a compact-type metric Lie algebra $\g=\s\oplus\z$. Its full spectral action is represented by derivations $D_\xi$ and vector fields $Y_\xi$ on $Z$. Define
\[
 G_{ab}=\ip{\nabla u_a}{\nabla u_b},\qquad
 B_{ab}=\ip{Y_a}{Y_b},\qquad \DD=G-B,
\]
and let $\kk=\{\xi:D_\xi=0\}$.
\begin{theorem}\label{thm:lossintro}
For every complete critical cluster, the derivations integrate to a continuous measure-preserving isometric action of $S\times\R^{\dim\z}$, where $S$ is the simply connected compact group with Lie algebra $\s$. Moreover,
\begin{equation}\label{eq:lossintro}
 \kk\subseteq\z,\qquad \DD\geq0,\qquad
 \rank\DD\leq2n-r\quad\mu\text{-a.e.},\qquad
 \dim\kk\leq\min\{r,2n-r\}.
\end{equation}
The last bound is attained for every $r$. If the critical gradients span the tangent module, then $\kk=\z$.
\end{theorem}

The tensor $\DD$ distinguishes loss of energy from disappearance of an action. A central direction can retain its full gradient energy on the approximating manifolds while acting trivially on every fixed spectral window of the limit. The bound in \eqref{eq:lossintro} is pointwise; no corresponding rank bound is asserted for the integral of $\DD$.

\subsection{Holomorphic recovery and its limitation}
Let $P^{\mathrm{hol}}$ be the $L^2$ orthogonal projection onto $H^0(M,K_M^{-1})$. For $q$ orthonormal real eigenfunctions with $1\leq\lambda_a\leq1+\eta$, $0<\eta<1$, let
\[
 W_a=2^{-1/2}(\nabla u_a-iJ\nabla u_a),\qquad
 s_I=W_{i_1}\wedge\cdots\wedge W_{i_n},\qquad
 \sigma_I=P^{\mathrm{hol}}s_I.
\]
\begin{theorem}\label{thm:recoveryintro}
The spectral anticanonical sections satisfy
\begin{equation}\label{eq:recoveryintro}
 \sum_{|I|=n}\norm{s_I-\sigma_I}_{L^2(\mu)}^2\leq C(n,q)\eta.
\end{equation}
If the limiting complex Gram matrix has rank $n$ almost everywhere, then the systems are nonzero for all sufficiently large indices. Their projective evaluations recover, in measure, the highest exterior power of the limiting complex Gram matrix.
\end{theorem}

A stronger conclusion about the moving part of this system is false. The following theorem concerns the actual orthogonal projections, not an arbitrary choice of holomorphic corrections.
\begin{theorem}\label{thm:counterintro}
There are smooth K\"ahler metrics $\omega_j$ on $X=\CP^1\times\CP^1$ such that
\[
 \Ric(\omega_j)\geq\omega_j,\qquad \lambda_4(\omega_j)\to1,
 \qquad \lambda_5(\omega_j)\geq\tfrac32,
 \qquad \Vol_{g_j}(X)\to0,
\]
with the following properties. The six projected determinants of the first four eigenfunctions form a basepoint-free system in $H^0(X,\OO(2,2))$. Its map $\Psi_j:X\to\CP^5$ has a two-dimensional image. Nevertheless, there are neighborhoods $U_j$ of the two collapsing ends such that $\mu_j(U_j)\to0$ and
\begin{equation}\label{eq:counterintro}
 \int_{U_j}\Psi_j^*\omega_{\FS}\wedge\alpha\longrightarrow2,
 \qquad \alpha=\pi_1^*\frac{\omega_1}{4\pi},
\end{equation}
where $\int_{\CP^1}\omega_{\FS}=1$. The metric measure limit is the product $\CP^1\times I$.
\end{theorem}

The defect in \eqref{eq:counterintro} is integral. A relative blow-up resolves the limiting base ideal and produces two end components of degree one. Thus a fixed divisor can appear in the limit even though every approximating system has no fixed part. This does not provide a nonproduct smooth collapsing limit. It shows that recovery of such limits cannot be based on a no-defect stability assertion for the moving spectral linear system.

The paper is organized as follows. Section~2 establishes the spectral compactness and bracket construction. Section~3 proves the kernel estimate and the dimension bound. Section~4 realizes all spectral symmetries and proves the loss estimates. Section~5 proves almost rigidity. Section~6 constructs the sharp examples. Sections~7 and~8 concern anticanonical recovery and its end defects. Section~9 explains the integer data carried by logarithmic tangent extensions and gives the conditional obstruction for holomorphic ruled fillings.

\section{Eigenfunctions and the critical spectral algebra}

\subsection{Bochner identities and normalized estimates}
All integrals in this section are taken with respect to $\mu=\omega^n/\int_M\omega^n$. Let $A_f$ be the self-adjoint endomorphism corresponding to the real Hessian of a real function $f$. Its complex-linear and complex-antilinear parts are
\[
 A_f^+=\tfrac12(A_f-JA_fJ),\qquad
 A_f^-=\tfrac12(A_f+JA_fJ).
\]
Thus $A_f^+J=JA_f^+$ and $A_f^-J=-JA_f^-$. With the unitary tensor norms,
\begin{equation}\label{eq:hessnorms}
 |\nabla f|^2=2|\partial f|^2,\qquad
 |A_f^-|^2=2|\nabla^{2,0}f|^2,\qquad
 |A_f^+|^2=2|\partial\dbar f|^2.
\end{equation}
The Hermitian metric on $T^{1,0}M$ is conjugate linear in its first argument. In particular,
\[
 W(v)=2^{-1/2}(v-iJv),\qquad |W(v)|^2=|v|^2.
\]

\begin{lemma}\label{lem:bochner}
Suppose that $\Ric(\omega)\geq\omega$ and put $\theta=\Ric(\omega)-\omega$. For every real smooth $f$,
\begin{equation}\label{eq:bochner}
 \int_M|\nabla^{2,0}f|^2\dd\mu+
 \int_M\theta(\nabla^{1,0}f,\nabla^{1,0}f)\dd\mu
 =\int_M(\Delta f)^2\dd\mu-\int_M|\partial f|^2\dd\mu.
\end{equation}
Moreover,
\begin{equation}\label{eq:mixedhessian}
 \int_M|\partial\dbar f|^2\dd\mu=\int_M(\Delta f)^2\dd\mu.
\end{equation}
Consequently $\lambda_1\geq1$. If $f=\sum_{a=1}^q t_a u_a$, where the $u_a$ are real orthonormal eigenfunctions with $1\leq\lambda_a\leq1+\eta$, the right-hand side of \eqref{eq:bochner} is at most $(1+\eta)\eta|t|^2$.
\end{lemma}
\begin{proof}
K\"ahler integration by parts and the commutation of one holomorphic and one antiholomorphic covariant derivative give
\[
 \int_M(\Delta f)^2\dd\mu
 =\int_M|\nabla^{2,0}f|^2\dd\mu+
 \int_M\Ric(\nabla^{1,0}f,\nabla^{1,0}f)\dd\mu.
\]
Splitting $\Ric=\omega+\theta$ proves \eqref{eq:bochner}. This is the general-eigenvalue form of the calculation in \cite[Lemma 2.1]{CWZ}. On an eigenfunction its right-hand side is $\lambda(\lambda-1)\norm{f}_2^2$. Orthogonality gives the assertion for a spectral sum. For $n\geq2$, Stokes' theorem applied to $(i\partial\dbar f)^2\wedge\omega^{n-2}$ proves \eqref{eq:mixedhessian}; its pointwise trace identity is a nonzero dimensional multiple of $(\Delta f)^2-|\partial\dbar f|^2$. For $n=1$ the identity is pointwise.
\end{proof}

For $X_f=J\nabla f$, the signs fixed in \eqref{eq:conventions} give
\begin{equation}\label{eq:hamiltonian}
 \iota_{X_f}\omega=-df,\qquad \diver X_f=0,\qquad
 |\Lie_{X_f}g|^2=4|A_f^-|^2.
\end{equation}
In particular, critical spectral directions are almost Killing in normalized $L^2$.

\begin{lemma}\label{lem:lowfrequency}
For each $\Lambda<\infty$, real eigenfunctions satisfying $-\Delta_{\R}f=\nu f$, $0<\nu\leq\Lambda$, and $\norm{f}_{L^2(\mu)}=1$ have uniformly bounded $L^\infty$ norms, gradient $L^\infty$ norms, and Hessian $L^2$ norms. The constants depend only on $n$ and $\Lambda$.
\end{lemma}
\begin{proof}
Bonnet--Myers gives $\operatorname{diam}M\leq D_n=\pi\sqrt{2n-1}$. Normalized Bishop--Gromov comparison then yields
\begin{equation}\label{eq:normalizedball}
 \mu(B_r(x))\geq b(n,r)>0\qquad(0<r\leq D_n).
\end{equation}
Put $U=\norm f_\infty$ and $F=|\nabla f|^2+2\nu f^2$. At a maximum of $F$, the real Bochner formula and $\Ric_g\geq0$ imply
\[
 0\geq\tfrac12\Delta_{\R}F\geq\nu|\nabla f|^2-2\nu^2f^2.
\]
Thus $\norm{\nabla f}_\infty\leq2\sqrt\nu\,U$. At a point where $|f|=U$, the ball of radius $r_\Lambda=\min\{1,(4\sqrt\Lambda)^{-1}\}$ satisfies $|f|\geq U/2$. Equation~\eqref{eq:normalizedball} and $\norm f_2=1$ bound $U$. Finally, integrated real Bochner gives
\[
 \int_M|\Hess f|^2\dd\mu
 =\int_M(\Delta_{\R}f)^2\dd\mu-
 \int_M\Ric_g(\nabla f,\nabla f)\dd\mu\leq\nu^2.
\]
None of these bounds contains an inverse unnormalized volume.
\end{proof}

A nonzero mean-zero eigenfunction also gives a lower diameter bound. It has a zero, and hence
\begin{equation}\label{eq:diameterlower}
 \norm f_\infty\leq\operatorname{diam}(M)\norm{\nabla f}_\infty,
 \qquad \operatorname{diam}(M)\geq\frac1{2\sqrt\nu}.
\end{equation}
This observation is useful when excluding a point limit.

\subsection{Compactness and the full spectral core}
We recall the precise compactness statements used below. The metric measure compactness, Sobolev convergence, and heat-flow theory are those of \cite{AGS,GMS,AH,AHT,AHPT}; the differential calculus is developed in \cite{Gigli}. An $\RCD(K,N)$ space is an infinitesimally Hilbertian metric measure space with the finite-dimensional synthetic Ricci lower bound. On such a space the Dirichlet form is
\[
 \mathcal E(f,h)=\int\ip{\nabla f}{\nabla h}\dd\mu,
 \qquad \mathcal E(f,h)=-\int(\Delta_\mu f)h\dd\mu.
\]
The Hessian on test functions is characterized by the polarized identity
\begin{align}\label{eq:weakHessian}
 2\Hess f(\nabla h,\nabla k)
 ={}&\ip{\nabla\ip{\nabla f}{\nabla h}}{\nabla k}
 +\ip{\nabla\ip{\nabla f}{\nabla k}}{\nabla h}
 -\ip{\nabla f}{\nabla\ip{\nabla h}{\nabla k}}.
\end{align}
A test function is bounded, has bounded weak gradient, and has a Laplacian in $W^{1,2}$. All identities for weak tensors are understood almost everywhere. The measure Laplacian is not identified with $\tr\Hess$ on a weighted collapsed limit.

\begin{proposition}\label{prop:compactness}
Let $\Ric(\omega_j)\geq\omega_j$. A subsequence of $(M_j,d_j,\mu_j)$ converges to a compact connected $\RCD(1,2n)$ probability space $(Z,d,\mu)$. If $Z$ is not a point, finite-index eigenvalues converge and orthonormal eigenbases can be chosen to converge uniformly and strongly in $W^{1,2}$ at every fixed index. More generally, eigenfunction convergence is asserted only in bounded spectral windows; a point limit has no nonconstant bounded-window eigenfunctions.

Scalar sequences bounded in the full $W^{1,2}$ norm have strongly $L^2$ convergent subsequences. Spectral functions with uniformly bounded frequencies have the bounds in Lemma~\ref{lem:lowfrequency}; their Hessians converge weakly, including contractions against strongly converging uniformly bounded test gradients. Finite spectral sums on $Z$ are dense in $C(Z)$ in the uniform norm and in $W^{1,2}(Z)$ in the energy norm.
\end{proposition}
\begin{proof}
The diameter and normalized ball bounds imply uniform covering estimates. Stability of the curvature-dimension condition and Mosco convergence of the energies give the asserted limit and spectral convergence \cite{GMS,AH}. Weak Hessian stability is \cite[Theorem 10.3 and Corollary 10.4, version 3]{AH}, applied with the bounds of Lemma~\ref{lem:lowfrequency}. The scalar compactness statement includes an $L^2$ bound, not just a Dirichlet-energy bound.

The heat kernel is continuous and its spectral expansion is uniform at each positive time \cite{AHT,AHPT}. For $f\in C(Z)$, first approximate $e^{-t\LL}f$ by finite spectral sums and then let $t\downarrow0$. Strong continuity on $C(Z)$ proves uniform density. Energy density follows from
\[
 \norm f_{W^{1,2}}^2=\sum_{\ell\geq0}(1+\nu_\ell)|\ip f{\phi_\ell}_{L^2}|^2.
\]
The same normalized heat bounds bound the number of eigenvalues in a fixed window. If a sequence has a nonzero bounded eigenvalue, \eqref{eq:diameterlower} excludes a point limit. If the entire sequence collapses to a point, its nonzero eigenvalues can instead tend to infinity; no finite nonconstant limiting eigenbasis is claimed in that case.
\end{proof}

For later use, if $f_j,h_j$ are convergent fixed-window spectral functions, then
\begin{equation}\label{eq:strongGram}
 \ip{\nabla f_j}{\nabla h_j}\longrightarrow
 \ip{\nabla f}{\nabla h}\quad\text{strongly in }L^2.
\end{equation}
Indeed, strong Sobolev convergence identifies the strong $L^1$ limit, and the uniform gradient bound upgrades it to $L^2$. Moreover,
\[
 |\nabla\ip{\nabla f_j}{\nabla h_j}|
 \leq |\Hess f_j|\,|\nabla h_j|+|\Hess h_j|\,|\nabla f_j|,
\]
so these scalar products are bounded in $W^{1,2}$.

\subsection{Poisson closure}
Suppose $\lambda_N(\omega_j)\to1$. Let $q$ be the multiplicity of the real eigenvalue two on the nonpoint limit. Spectral convergence gives $q\geq N$ and real orthonormal eigenfunctions
\begin{equation}\label{eq:cluster}
 -\Delta_j u_{j,a}=\lambda_{j,a}u_{j,a},\qquad
 1\leq\lambda_{j,a}\leq1+\eta_j\quad(1\leq a\leq q),\qquad \eta_j\to0,
\end{equation}
with
\begin{equation}\label{eq:clustergap}
 \lambda_{j,q+1}\geq1+\gamma
\end{equation}
for some $\gamma>0$ and all large $j$. The entire cluster is used. A proper subspace of it need not be closed under limiting brackets.

\begin{lemma}\label{lem:poisson}
For real unit eigenfunctions $-\Delta f=\lambda f$ and $-\Delta h=\nu h$ with $1\leq\lambda,\nu\leq1+\eta$, put $p(f,h)=\ip{J\nabla f}{\nabla h}$. Then $\int p(f,h)\dd\mu=0$ and
\begin{align}\label{eq:poissonidentity}
 (\Delta+1)p(f,h)
 ={}&-(\lambda+\nu-2)p(f,h)+\theta^{\R}(J\nabla f,\nabla h)
 +\sum_{\ell=1}^{2n}\ip{JA_f^-e_\ell}{A_h^-e_\ell},
\end{align}
where $\theta^{\R}=\Ric_g-g$. In particular,
\begin{equation}\label{eq:poissonL1}
 \norm{(\Delta+1)p(f,h)}_1\leq8(1+\eta)\eta.
\end{equation}
\end{lemma}
\begin{proof}
The mean vanishes by \eqref{eq:hamiltonian}. In a normal orthonormal frame, use $\nabla J=0$ and
$\Delta_{\R}^{\nabla}\nabla f=\nabla\Delta_{\R}f+\Ric^\sharp\nabla f$ to obtain
\[
 \Delta_{\R}p=-2(\lambda+\nu)p+2\Ric_g(J\nabla f,\nabla h)
 +2\sum_\ell\ip{JA_fe_\ell}{A_he_\ell}.
\]
The mixed complex-linear and complex-antilinear terms are orthogonal. The term involving $A_f^+,A_h^+$ is zero: these are Hermitian endomorphisms, the complex trace of their product is real, and multiplication by $J$ has zero real trace on that product. Dividing by two proves \eqref{eq:poissonidentity}.

The first term has $L^1$ norm at most $4(1+\eta)\eta$, since $\norm p_1\leq2\sqrt{\lambda\nu}$. Positivity and $J$-invariance of $\theta^{\R}$, together with Lemma~\ref{lem:bochner}, bound its mixed term by $2(1+\eta)\eta$. Equation~\eqref{eq:hessnorms} gives the same bound for the last term.
\end{proof}

Let $\mathsf P_j$ be the $L^2$ projection onto the cluster, and write
\begin{equation}\label{eq:bracketprojection}
 p_j(u_{j,a},u_{j,b})=\sum_c c_{ab}^{\ c}(j)u_{j,c}+r_{j,ab},
 \qquad r_{j,ab}\perp E_j.
\end{equation}
Both the Poisson bracket and its projection are uniformly bounded. The remainder has mean zero. Setting $A_j=-\Delta_j$, the gap gives
\[
 \gamma\norm{r_{j,ab}}_2^2
 \leq\ip{(A_j-1)r_{j,ab}}{r_{j,ab}}
 =\ip{(A_j-1)p_j(u_{j,a},u_{j,b})}{r_{j,ab}}
 \leq C\eta_j.
\]
Adding $\norm r_2^2$ controls its energy. Thus
\begin{equation}\label{eq:remainder}
 \norm{r_{j,ab}}_{W^{1,2}}\longrightarrow0.
\end{equation}
After a further subsequence, the bounded coefficients converge. Define
\[
 [e_a,e_b]=\sum_c c_{ab}^{\ c}e_c\quad\text{on }\R^q.
\]

\begin{proposition}\label{prop:criticalalgebra}
This bracket makes $E\simeq\R^q$ a Lie algebra with an ad-invariant positive definite inner product. It is the orthogonal direct sum
\[
 \g=\s\oplus\z,
\]
where $\s$ is compact semisimple and $\z$ is the center.
\end{proposition}
\begin{proof}
Divergence-free integration by parts gives
$\int p(f,h)k\dd\mu=-\int h\,p(f,k)\dd\mu$.
Hence $c_{abc}(j)$ is fully alternating. Pair the ordinary Poisson Jacobi identity with $u_{j,d}$ and use \eqref{eq:bracketprojection}. The result is
\[
 \sum_{\mathrm{cyc}(a,b,c)}
 \left(\sum_e c_{ab}^{\ e}(j)c_{ec}^{\ d}(j)
       +\ip{r_{j,ab}}{r_{j,cd}}_{L^2}\right)=0.
\]
The remainders disappear by \eqref{eq:remainder}. This proves Jacobi and ad-invariance. For completeness, ad-invariance makes the orthogonal complement of $[\g,\g]$ equal to the center. The adjoint operators are skew-adjoint; the center-free derived algebra is compact semisimple, giving the stated decomposition.
\end{proof}

On the limit define
\begin{equation}\label{eq:GP}
 G_{ab}=\ip{\nabla u_a}{\nabla u_b},\qquad
 P_{ab}=u_{[e_a,e_b]}.
\end{equation}
The complex Gram matrices $W_j^*W_j$ converge strongly in $L^2$ to $H=G+iP$. Therefore
\begin{equation}\label{eq:complexrank}
 H=G+iP\geq0,\qquad \rank_\C H\leq n\quad\mu\text{-a.e.}
\end{equation}
The entries of $G$ belong to $W^{1,2}\cap L^\infty$. These statements require no complex structure on $Z$.

\section{Kernel projections and the dimension bound}

\subsection{A uniform projection estimate}
The estimate below is independent of a lower bound for the volume or for a nonzero singular value of the gradient matrix. It is proved on the smooth manifolds before any limit is taken.

\begin{lemma}\label{lem:projection}
Let $W:\C^k\to T^{1,0}M$ be a smooth bundle map on a closed K\"ahler manifold with $\Ric_g\geq g$. Suppose
\begin{equation}\label{eq:Wequation}
 \Delta_{\R}^{\nabla}W=-2W+\Ric^\sharp W+\mathcal R.
\end{equation}
For $\tau>0$, set
\[
 H=W^*W,\quad F=(H+\tau I)^{-1},\quad S=WFW^*,\quad \Pi_\tau=\tau F.
\]
Then $0\leq S\leq I$, and
\begin{equation}\label{eq:projectionenergy}
 \int_M|\nabla\Pi_\tau|^2\dd\mu
 \leq4k+\frac8\tau\norm{\nabla W}_2\norm{\nabla^{0,1}W}_2
       +\frac4\tau\norm W_2\norm{\mathcal R}_2.
\end{equation}
\end{lemma}
\begin{proof}
The nonzero eigenvalues of $S$ are $h/(h+\tau)$, where $h$ ranges over the positive eigenvalues of $H$. In a unitary frame, put
\[
 A_i=\nabla_iW,\quad C_i=\nabla_{\bar i}W,\quad
 Y_i=F^{1/2}W^*A_iF^{1/2},\quad
 Z_i=F^{1/2}C_i^*WF^{1/2}.
\]
Define the nonnegative quantity
\begin{equation}\label{eq:projectiondefect}
 \mathcal Q_\tau=\sum_i\left(
 \norm{(I-S)^{1/2}A_iF^{1/2}}_\HS^2+
 \norm{(I-S)^{1/2}C_iF^{1/2}}_\HS^2\right).
\end{equation}
Since the domain bundle is trivial,
\[
 \nabla_iH=W^*A_i+C_i^*W,\qquad
 \nabla_{\bar i}H=(\nabla_iH)^*.
\]
The matrix differentiation formula yields
\[
 \tfrac12\Delta_{\R}\log\det(H+\tau I)
 =\tfrac12\tr(F\Delta_{\R}H)
   -\sum_i\tr\bigl(F(\nabla_iH)F(\nabla_{\bar i}H)\bigr).
\]
Using \eqref{eq:Wequation}, the first term is
\begin{align*}
 \tfrac12\tr(F\Delta_{\R}H)
 ={}&-2\tr(FH)+\tr(FW^*\Ric^\sharp W)
       +\operatorname{Re}\tr(FW^*\mathcal R)\\
 &+\sum_i\tr\bigl(F(A_i^*A_i+C_i^*C_i)\bigr).
\end{align*}
Subtracting $\sum_i(|Y_i|^2+|Z_i|^2)$ from the last line gives exactly $\mathcal Q_\tau$. Thus
\begin{align}\label{eq:logdet}
 \tfrac12\Delta_{\R}\log\det(H+\tau I)
 ={}&\mathcal Q_\tau-2\sum_i\operatorname{Re}\ip{Y_i}{Z_i}_\HS-2\tr(FH)\notag\\
 &+\tr(FW^*\Ric^\sharp W)+\operatorname{Re}\tr(FW^*\mathcal R).
\end{align}
Integrating and using $\Ric^\sharp\geq I$ and $\tr(FH)\leq k$, we obtain
\begin{equation}\label{eq:Qbound}
 \int_M\mathcal Q_\tau\dd\mu
 \leq k+\frac2\tau\norm{\nabla W}_2\norm{\nabla^{0,1}W}_2
       +\frac1\tau\norm W_2\norm{\mathcal R}_2.
\end{equation}
Indeed, $\norm{F^{1/2}W^*}_\op\leq1$ and $\norm{F^{1/2}}_\op\leq\tau^{-1/2}$ give $|Y_i|\leq\tau^{-1/2}|A_i|$ and $|Z_i|\leq\tau^{-1/2}|C_i|$.

It remains to estimate the projection itself. The resolvent derivative is
\[
 \nabla_i\Pi_\tau=-\tau F(W^*A_i+C_i^*W)F.
\]
Singular-value decomposition of $W$ gives
\begin{align*}
 \norm{\tau FW^*A_iF}_\HS^2
 &\leq\norm{(I-S)^{1/2}A_iF^{1/2}}_\HS^2,\\
 \norm{\tau FC_i^*WF}_\HS^2
 &\leq\norm{(I-S)^{1/2}C_iF^{1/2}}_\HS^2.
\end{align*}
To check the first inequality, let $h_a,h_b$ be the squared singular values in the corresponding row and column. The ratio of the coefficient on the left to that on the right is
\[
 \frac{\tau h_a}{(h_a+\tau)(h_b+\tau)}\leq1.
\]
Rows orthogonal to the image of $W$ give zero on the left. The second inequality is obtained by interchanging the roles of the row and column. Since $\Pi_\tau$ is Hermitian,
\[
 |\nabla\Pi_\tau|^2=2\sum_i|\nabla_i\Pi_\tau|^2\leq4\mathcal Q_\tau.
\]
Together with \eqref{eq:Qbound}, this proves \eqref{eq:projectionenergy}.
\end{proof}

For $k$ fixed orthonormal real combinations $f_{j,1},\ldots,f_{j,k}$ of \eqref{eq:cluster}, use the columns $W(\nabla f_{j,\alpha})$. Gradient commutation gives \eqref{eq:Wequation} with $\norm{\mathcal R_j}_2\leq C\eta_j$. Lemma~\ref{lem:bochner} gives $\norm{\nabla^{0,1}W_j}_2\leq C\sqrt{\eta_j}$. The other norms are bounded by Lemma~\ref{lem:lowfrequency}. Hence
\begin{equation}\label{eq:spectralprojection}
 \int|\nabla\Pi_{j,\tau}|^2\dd\mu_j
 \leq4k+C(n,k)\tau^{-1}\sqrt{\eta_j}.
\end{equation}
The order of limits matters: $\tau$ is fixed when $j\to\infty$. No bound uniform in $j$ and $\tau$ simultaneously is asserted.

\subsection{Independence of commuting gradients}
\begin{theorem}\label{thm:commuting}
Let $\mathfrak a\subset\g$ be a $k$-dimensional real subspace with $[\mathfrak a,\mathfrak a]=0$, and let $f_1,\ldots,f_k\in E$ correspond to an orthonormal basis. Then
\[
 \rank\bigl(\ip{\nabla f_\alpha}{\nabla f_\beta}\bigr)_{\alpha,\beta=1}^k=k
 \quad\mu\text{-a.e.}
\]
In particular, $k\leq n$.
\end{theorem}
\begin{proof}
Take the corresponding fixed combinations on $M_j$. Their Poisson brackets converge to zero in $W^{1,2}$, by \eqref{eq:remainder}. On a smooth K\"ahler manifold,
\begin{equation}\label{eq:gradientbracket}
 [\nabla f,\nabla h]+[J\nabla f,J\nabla h]
 =2A_h^-\nabla f-2A_f^-\nabla h.
\end{equation}
Also $[J\nabla f,J\nabla h]=J\nabla p(f,h)$ with our Hamiltonian sign convention. Thus the first bracket tends to zero in $L^2$. Pairing with strongly converging bounded test gradients and applying weak Hessian stability gives
\begin{equation}\label{eq:commutingHessian}
 \Hess f_\alpha(\nabla f_\beta,\cdot)
 =\Hess f_\beta(\nabla f_\alpha,\cdot)\quad\mu\text{-a.e.}
\end{equation}
One first uses a countable family of test gradients generating the cotangent module; density then gives the tensor identity.

Write $G$ for the $k\times k$ Gram matrix in this proof and let $\Pi_0(x)$ be the orthogonal projection onto $\ker G(x)$. The entries of $G$ are bounded Sobolev functions. Equation~\eqref{eq:commutingHessian} gives
\[
 dG_{\alpha\beta}=2\Hess f_\alpha(\nabla f_\beta,\cdot),\qquad
 (dG)\Pi_0=\Pi_0(dG)=0.
\]
The second identity follows because a coefficient vector in $\ker G(x)$ gives the zero vector $\sum_\beta v_\beta\nabla f_\beta(x)$. No derivative of that coefficient vector is taken.

The complex Gram matrices of the approximants converge strongly to this real matrix $G$. For fixed $\tau>0$, the regularized projections converge strongly to
\[
 \Pi_\tau=\tau(G+\tau I)^{-1}.
\]
By lower semicontinuity in \eqref{eq:spectralprojection},
\[
 \int_Z|d\Pi_\tau|^2\dd\mu\leq4k\qquad(\tau>0).
\]
As $\tau\downarrow0$, these matrices converge pointwise and strongly in $L^2$ to $\Pi_0$. Consequently $\Pi_0\in W^{1,2}$.

Differentiate $G\Pi_0=0$. The preceding annihilation identity gives $G\,d\Pi_0=0$, so $\Pi_0d\Pi_0=d\Pi_0$. Transposing gives $(d\Pi_0)\Pi_0=d\Pi_0$. Differentiating $\Pi_0^2=\Pi_0$ now yields $d\Pi_0=2d\Pi_0$, and hence $d\Pi_0=0$. These product rules are legitimate because both matrix factors are bounded Sobolev functions.

The Poincar\'e inequality on the connected space makes $\Pi_0$ a constant matrix. Since
\[
 \int_ZG_{\alpha\beta}\dd\mu
 =\ip{\LL f_\alpha}{f_\beta}_{L^2}=2\delta_{\alpha\beta},
\]
integrating $G\Pi_0=0$ gives $\Pi_0=0$. Finally, the corresponding complex Gram matrix has zero imaginary part, so \eqref{eq:complexrank} implies $k\leq n$.
\end{proof}

\begin{corollary}\label{cor:levelsets}
Every fixed nonzero $u\in E$ satisfies $|\nabla u|>0$ almost everywhere. For every $a\in\R$, $\mu\{u=a\}=0$.
\end{corollary}
\begin{proof}
Apply Theorem~\ref{thm:commuting} to the one-dimensional subspace generated by $u$. The locality of weak gradients implies $|\nabla u|=0$ almost everywhere on a level set. No unique-continuation theorem for arbitrary RCD eigenfunctions is used.
\end{proof}

\subsection{Hermitian rank and essential dimension}
We record the linear algebra used in the remaining proofs.
\begin{lemma}\label{lem:rank}
Let $G$ be real symmetric, $P$ real skew-symmetric, and $G+iP\geq0$.
\begin{enumerate}[label=(\roman*),leftmargin=2.5em]
\item $\ker_\R G\subseteq\ker_\R P$ and $\rank_\R G\leq2\rank_\C(G+iP)$.
\item Suppose a real $c\times c$ block $A$ of $G$ is positive definite and the corresponding rows and columns of $P$ vanish. If $\rank_\R P=2m$, then
\[
 \rank_\R G\geq c+2m,\qquad
 \rank_\C(G+iP)\geq c+m.
\]
\item If $\rank_\R P=2n$ and $\rank_\C(G+iP)\leq n$, then $\rank_\R G=2n$ and $\ker_\R G=\ker_\R P$. On the real quotient, $P$ determines an orthogonal complex structure for the inner product $G$.
\end{enumerate}
\end{lemma}
\begin{proof}
If $Gv=0$ for real $v$, then $v^*(G+iP)v=0$, and positivity gives $(G+iP)v=0$. Also $2G=(G+iP)+(G-iP)$, which proves the rank inequality in (i).

For (ii), write
\[
 G=\begin{pmatrix}A&B\\B^T&C\end{pmatrix},\qquad
 P=\begin{pmatrix}0&0\\0&Q\end{pmatrix}.
\]
The Schur complement $S=C-B^TA^{-1}B$ satisfies $S+iQ\geq0$. By (i), $\ker S\subseteq\ker Q$, so $\rank G=c+\rank S\geq c+2m$. Moreover, $2iQ=(S+iQ)-\overline{(S+iQ)}$ implies $\rank_\C(S+iQ)\geq m$. This gives the complex-rank assertion.

For (iii), the two rank bounds force $\rank G=\rank P=2n$ and equality of the kernels. On the quotient, make $G$ the identity. The real skew matrix has two-dimensional blocks with parameters $s_\alpha\geq0$. Positivity gives $s_\alpha\leq1$, while complex rank at most $n$ forces every block of $I+iP$ to have rank one. Thus all $s_\alpha=1$.
\end{proof}

\begin{proof}[Proof of Theorem~\ref{thm:dimensionintro}]
If $q=0$, there is nothing to prove. Write $\g=\z\oplus\s$, $c=\dim\z$ and $d=\dim\s$. The center is commuting, so its Gram block is positive definite almost everywhere by Theorem~\ref{thm:commuting}. Its Poisson rows and columns are zero. At a point where $\rank P=2m$, Lemma~\ref{lem:rank} gives
\begin{equation}\label{eq:realbudget}
 c+2m\leq\rank G\leq r.
\end{equation}

Decompose $\s$ into its finitely many simple ideals and fix a nonzero vector $\xi_\alpha$ in each. By Corollary~\ref{cor:levelsets}, $u_{\xi_\alpha}(x)\neq0$ on a common set of full measure. At such a point, the evaluation covector $\ell_x\in\s^*$ has nonzero component in every simple dual factor. The simply connected compact group with Lie algebra $\s$ acts on the coadjoint orbit of $\ell_x$ with injective infinitesimal action. Indeed, a nonzero infinitesimal kernel would contain a simple ideal, forcing the corresponding evaluation component to be zero.

The orbit has real dimension $2m=\rank P(x)$ and its canonical invariant K\"ahler structure \cite[Section 1]{Rieffel}. The isotropy acts faithfully on its complex $m$-dimensional tangent space, hence embeds infinitesimally in $\mathfrak u(m)$. An isometry fixing a point with identity differential is the identity, which justifies this faithfulness. Therefore
\begin{equation}\label{eq:semisimplebudget}
 d\leq2m+m^2.
\end{equation}
This is also the dimension estimate of \cite[Proposition 3.1]{CWZ}, applied only to the smooth auxiliary orbit. If $\s=0$, read it with $m=d=0$.

Combining \eqref{eq:realbudget} and \eqref{eq:semisimplebudget},
\[
 q=c+d\leq r+m^2\leq r+\lfloor r/2\rfloor^2.
\]
The examples proving sharpness in every dimension are constructed in Section~\ref{sec:examples}. The argument does not assume that the rank of $P$ is already constant.
\end{proof}

\begin{corollary}\label{cor:dimensionequality}
If $r=2k+\epsilon$, $\epsilon\in\{0,1\}$, and $q=Q(r)$, then
\[
 \dim\z=\epsilon,\qquad \dim\s=k^2+2k,\qquad
 \rank P=2k,\qquad\rank G=r\quad\mu\text{-a.e.}
\]
In particular, a volume-collapsing sequence with $\lambda_{n^2}\to1$ has essential and Hausdorff dimensions $2n-1$, critical multiplicity $n^2$, and a one-dimensional center. Its complex Gram matrix has rank $n$ almost everywhere.
\end{corollary}
\begin{proof}
Equality in the chain $q\leq c+m^2+2m\leq r+m^2\leq r+k^2$ forces all the asserted equalities. For the last statement, the collapse dimension gap \cite[Theorem 1.2]{DPG} gives $\dim_{\HH}Z\leq2n-1$. Spectral convergence gives $q\geq n^2$, while $Q(2n-2)=n^2-1$ and $Q(2n-1)=n^2$. Hence $q=n^2$ and $r=2n-1$. The general inequality $r\leq\dim_{\HH}Z$ gives equality of dimensions here. Finally, Lemma~\ref{lem:rank}(ii) gives complex rank at least $1+(n-1)=n$, and \eqref{eq:complexrank} gives the reverse bound.
\end{proof}

For high multiplicity we will also use the complex dimension budget $c+m\leq n$, rather than just \eqref{eq:realbudget}.
\begin{corollary}\label{cor:centerexclusion}
If $q>n^2$, then $\z=0$ and
\[
 \rank P=\rank G=2n,\qquad \ker G=\ker P\quad\mu\text{-a.e.}
\]
\end{corollary}
\begin{proof}
Lemma~\ref{lem:rank}(ii) and \eqref{eq:complexrank} give $c+m\leq n$. If $c\geq1$, then
\[
 q\leq(n-c)^2+2(n-c)+c
 =n^2-(c-1)(2n-c)\leq n^2,
\]
a contradiction. Thus $c=0$. If $m\leq n-1$, \eqref{eq:semisimplebudget} gives $q\leq n^2-1$, again impossible. Hence $m=n$, and Lemma~\ref{lem:rank}(iii) completes the proof.
\end{proof}

\section{Full spectral actions and symmetry loss}

\subsection{Derivations on every eigenspace}
No lower bound on $q$ is imposed in this section. Let $\{\phi_\ell\}_{\ell\geq0}$ be a real orthonormal eigenbasis of $\LL$ on $Z$, with $\LL\phi_\ell=\nu_\ell\phi_\ell$ and $\phi_0=1$. Choose corresponding convergent eigenfunctions on $M_j$. Set $X_{j,a}=J_j\nabla u_{j,a}$.

For each fixed $a,\ell$,
\[
 |\nabla(X_{j,a}\phi_{j,\ell})|
 \leq |\nabla X_{j,a}|\,|\nabla\phi_{j,\ell}|
       +|X_{j,a}|\,|\Hess\phi_{j,\ell}|.
\]
The right-hand side is bounded in $L^2$. A diagonal subsequence therefore defines
\begin{equation}\label{eq:Ddefinition}
 D_a\phi_\ell=L^2\text{-}\lim_j X_{j,a}\phi_{j,\ell}.
\end{equation}
Let $\FF$ be the algebraic direct sum of all real eigenspaces. Extend $D_\xi=\sum_a\xi_aD_a$ linearly on $\FF$.

\begin{lemma}\label{lem:Dproperties}
The operators satisfy
\begin{align}
 D_a\ker(\LL-\nu)&\subseteq\ker(\LL-\nu),\label{eq:Dpreserve}\\
 \ip{D_af}{h}_{L^2}&=-\ip f{D_ah}_{L^2},\label{eq:Dskew}\\
 \norm{D_\xi f}_2&\leq C_n|\xi|\norm{\nabla f}_2.\label{eq:Denergy}
\end{align}
Consequently they extend continuously from $W^{1,2}$ to $L^2$. For $f,h\in\FF$,
\begin{equation}\label{eq:DLeibniz}
 D_a(fh)=fD_ah+hD_af.
\end{equation}
The same identity holds if $f\in W^{1,2}$ and $h\in\FF$. On $\FF$,
\begin{equation}\label{eq:DLie}
 [D_a,D_b]=\sum_c c_{ab}^{\ c}D_c,\qquad
 D_au_b=\sum_c c_{ab}^{\ c}u_c.
\end{equation}
\end{lemma}
\begin{proof}
For a divergence-free smooth field $X$ and real eigenfunctions $\LL_j f=\nu f$, $\LL_j h=\sigma h$, integration by parts gives
\begin{equation}\label{eq:commutatorenergy}
 (\sigma-\nu)\int(Xf)h\dd\mu_j
 =\int(\Lie_Xg_j)(\nabla f,\nabla h)\dd\mu_j.
\end{equation}
This follows by integrating $X\ip{\nabla f}{\nabla h}$ and comparing the differentiated gradients. For $X=X_{j,a}$ the right-hand side tends to zero by \eqref{eq:hamiltonian} and Lemma~\ref{lem:bochner}. Testing against the complete limiting eigenbasis gives \eqref{eq:Dpreserve}. Divergence-free integration gives \eqref{eq:Dskew}. The uniform bound for $X_{j,\xi}$ gives \eqref{eq:Denergy}; energy density then gives the unique extension.

For the product rule, let $f_j,h_j$ be the corresponding finite spectral sums and set $v_j=f_jh_j$. The formula
\[
 \LL_jv_j=(\LL_jf_j)h_j+f_j(\LL_jh_j)-2\ip{\nabla f_j}{\nabla h_j}
\]
gives a uniform $L^2$ bound. Choose $R$ outside the spectrum of $\LL$, and let $v_{j,R}$ be the projection onto frequencies at most $R$. Spectral orthogonality gives
\begin{equation}\label{eq:spectraltail}
 \norm{\nabla(v_j-v_{j,R})}_2^2\leq R^{-1}\norm{\LL_jv_j}_2^2.
\end{equation}
For fixed $R$, the coefficients converge and \eqref{eq:Ddefinition} applies to $v_{j,R}$. The smooth product rule gives
\[
 X_{j,a}(f_jh_j)=f_jX_{j,a}h_j+h_jX_{j,a}f_j
 \longrightarrow fD_ah+hD_af
\]
strongly in $L^2$. First let $j\to\infty$ and then $R\to\infty$ in \eqref{eq:spectraltail}. Equation~\eqref{eq:Denergy} identifies the result with $D_a(fh)$. Approximating a general $f\in W^{1,2}$ by finite spectral sums proves the stated extension; multiplication by the fixed bounded Lipschitz $h$ is continuous in $W^{1,2}$.

The second identity of \eqref{eq:DLie} is \eqref{eq:bracketprojection}. For the first, all compositions are well-defined on $\FF$ by \eqref{eq:Dpreserve}. For $f,h\in\FF$,
\[
 \ip{[D_a,D_b]f}{h}
 =-\ip{D_bf}{D_ah}+\ip{D_af}{D_bh}.
\]
Each product is a strong $L^2$ limit. On $M_j$,
$[X_{j,a},X_{j,b}]=J_j\nabla p_j(u_{j,a},u_{j,b})$.
The remainder applied to $f_j$ has $L^2$ norm bounded by
$\norm{\nabla r_{j,ab}}_2\norm{\nabla f_j}_\infty\to0$.
Testing with the full spectral core proves the identity.
\end{proof}

\begin{proposition}\label{prop:globalaction}
Let $S$ be the simply connected compact group with Lie algebra $\s$, and let
$\widetilde G=S\times(\z,+)$. The derivations integrate to a continuous right action of $\widetilde G$ on $Z$ by measure-preserving isometries. Its infinitesimal kernel is contained in $\z$. The restriction to $S$ has finite global kernel. For the integrated action,
\begin{equation}\label{eq:equivariance}
 u_\eta(x\cdot g)=u_{\Ad_g\eta}(x).
\end{equation}
\end{proposition}
\begin{proof}
Each eigenspace is finite-dimensional and carries an orthogonal Lie algebra representation by Lemma~\ref{lem:Dproperties}. It integrates to the simply connected group $\widetilde G$. The Hilbert direct sum gives unitary operators $T_g$ satisfying
\[
 T_{gh}=T_gT_h,\qquad T_g\LL=\LL T_g,\qquad T_g1=1.
\]
They preserve the energy norm. The generator on a one-parameter subgroup is the skew-adjoint direct sum of its finite-dimensional generators; \eqref{eq:Denergy} puts $W^{1,2}$ in its domain.

For $f,h\in\FF$, differentiate in $L^2$ the expression
\[
 T_{\exp(-t\xi)}\bigl((T_{\exp(t\xi)}f)(T_{\exp(t\xi)}h)\bigr).
\]
The factors remain in fixed finite spectral spaces, and their product is in $W^{1,2}$. Equation~\eqref{eq:DLeibniz} makes the derivative zero. Thus $T_g(fh)=(T_gf)(T_gh)$ for every exponential, and hence for every $g$ in the connected group.

For fixed $h\in\FF$, approximate any $f\in L^2$ by finite spectral sums. Since $h$ and $T_gh$ are bounded, the same identity holds for this $f$. Iterating with powers of a finite spectral function gives
\[
 \int|T_gf|^{2m}\dd\mu=\int|f|^{2m}\dd\mu.
\]
Letting $m\to\infty$ proves equality of supremum norms. Here continuity and full support identify essential and ordinary suprema. Uniform density of $\FF$ now extends $T_g$ to a unital algebra automorphism of $C(Z)$. It is induced by a unique homeomorphism $\Phi_g$, with $T_gf=f\circ\Phi_g$. Strong continuity on the finite eigenspaces and the supremum-norm isometry imply strong continuity on $C(Z)$, hence joint continuity of the action. Preservation of integrals gives $(\Phi_g)_\#\mu=\mu$.

For bounded Sobolev functions, the energy measure is determined by
\[
 \int h|\nabla f|^2\dd\mu
 =\mathcal E(f,hf)-\tfrac12\mathcal E(f^2,h).
\]
Multiplicativity and energy invariance imply
\begin{equation}\label{eq:energyequivariance}
 |\nabla T_gf|^2=T_g(|\nabla f|^2)\quad\mu\text{-a.e.}
\end{equation}
At this stage $T_g$ is already composition with a measure-preserving homeomorphism, so this identity is meaningful also for bounded measurable right-hand sides. Approximation proves it for all Lipschitz $f$. Sobolev-to-Lipschitz then shows that $T_g$ preserves the class of $1$-Lipschitz functions. Applying this also to $g^{-1}$ proves that $\Phi_g$ is an isometry.

Equation~\eqref{eq:equivariance} follows from \eqref{eq:DLie}. If $D_\xi=0$, then $u_{[\xi,\eta]}=0$ for all $\eta$, and hence $\xi\in\z$. The kernel on the compact group $S$ has zero Lie algebra and is finite. No compactness of the factor $(\z,+)$ is needed in the preceding construction.
\end{proof}

\subsection{The loss matrix}
We give details of the module realization because a global $L^2$ operator bound alone would not suffice for the pointwise conclusions.

\begin{lemma}\label{lem:vectorrealization}
There are vector fields $Y_\xi\in L^\infty(TZ)$, linear in $\xi$, such that
\[
 D_\xi f=\ip{Y_\xi}{\nabla f}\quad(f\in W^{1,2}),\qquad
 |Y_\xi|^2\leq|\nabla u_\xi|^2\quad\mu\text{-a.e.}
\]
In particular, the matrix $B_{ab}=\ip{Y_a}{Y_b}$ satisfies $0\leq B\leq G$.
\end{lemma}
\begin{proof}
For finitely many spectral functions $f_i$ and bounded coefficients $a_i$, the smooth inequality
\[
 \left|\sum_i a_iX_{j,\xi}f_{j,i}\right|^2
 \leq|\nabla u_{j,\xi}|^2
       \left|\sum_i a_i\nabla f_{j,i}\right|^2
\]
passes to the limit. Initially take coefficients obtained from bounded continuous approximations in a common realization of the measured convergence. Strong scalar convergence in \eqref{eq:Ddefinition} and \eqref{eq:strongGram} proves the integrated localized inequality. Bounded measurable coefficients follow by approximation. The resulting assignment
$\sum_i a_i\,df_i\mapsto\sum_i a_iD_\xi f_i$
is well-defined on the generated cotangent module and bounded by $|\nabla u_\xi|$. Finite spectral functions are energy-dense, so it extends to the whole cotangent module. The Hilbert-module duality of \cite{Gigli} represents it by $Y_\xi$. Taking rational $\xi$ first and then continuity proves the simultaneous quadratic inequality $B\leq G$.
\end{proof}

\begin{proposition}\label{prop:lossrank}
Let $r$ be the essential dimension. Then
\begin{equation}\label{eq:lossrank}
 \DD:=G-B\geq0,
 \qquad \rank\DD\leq2n-r\quad\mu\text{-a.e.}
\end{equation}
\end{proposition}
\begin{proof}
A countable collection of gradients of finite spectral functions generates the tangent module. On countably many measurable pieces, choose $r$ of them, $\nabla\psi_1,\ldots,\nabla\psi_r$, as a basis. On each piece set
\[
 A_{\alpha\beta}=\ip{\nabla\psi_\alpha}{\nabla\psi_\beta},\qquad
 C_{a\alpha}=D_a\psi_\alpha.
\]
Here $A>0$, and expansion of $Y_a$ in this basis gives
\begin{equation}\label{eq:visibleSchur}
 B=CA^{-1}C^T.
\end{equation}
On the smooth approximants, the joint Gram matrix of the $r+q$ vectors
\[
 \nabla\psi_{j,1},\ldots,\nabla\psi_{j,r},
 \quad J_j\nabla u_{j,1},\ldots,J_j\nabla u_{j,q}
\]
is
\[
 \begin{pmatrix}A_j&C_j^T\\C_j&G_j\end{pmatrix}\geq0
\]
and has real rank at most $2n$.
All its entries converge strongly in $L^2$ and are uniformly bounded. Positivity and the vanishing of all $(2n+1)$-minors therefore pass to the limit. Taking the Schur complement of the positive block $A$ gives
\[
 G-CA^{-1}C^T\geq0,\qquad
 r+\rank(G-CA^{-1}C^T)\leq2n.
\]
Equation~\eqref{eq:visibleSchur} proves the assertion. The inverse is used only pointwise on a full-rank measurable piece; it is never differentiated.
\end{proof}

\begin{proof}[Proof of Theorem~\ref{thm:lossintro}]
The action and the inclusion $\kk\subseteq\z$ are Proposition~\ref{prop:globalaction}. The tensor estimate is Proposition~\ref{prop:lossrank}. Choose an orthonormal basis of $\kk$. Since it is central, Theorem~\ref{thm:commuting} gives a positive definite Gram block $G_{\kk\kk}$ almost everywhere. Its vector fields $Y_\xi$ are zero, so
\[
 \DD_{\kk\kk}=G_{\kk\kk}>0.
\]
Thus $\dim\kk\leq\rank\DD\leq2n-r$. Independence of the same gradients gives $\dim\kk\leq r$.

If all critical gradients span $TZ$, a central $Y_\xi$ is orthogonal to each of them by \eqref{eq:DLie}, and hence is zero. This proves $\kk=\z$ under the stated hypothesis. The sharp examples for all $r$ are given in Section~\ref{sec:examples}.
\end{proof}

\begin{corollary}\label{cor:extremecenter}
For a volume-collapsing sequence with $\lambda_{n^2}\to1$, precisely one critical Lie algebra direction acts trivially on the entire limit spectral resolution. It is the center. The semisimple part has dimension $n^2-1$, and its auxiliary coadjoint orbit attains the maximal holomorphic isometry dimension in complex dimension $n-1$.
\end{corollary}
\begin{proof}
Corollary~\ref{cor:dimensionequality} gives a one-dimensional center and $\rank G=2n-1=r$, so Theorem~\ref{thm:lossintro} gives $\kk=\z$. The dimension assertion is the equality case of \eqref{eq:semisimplebudget}.
\end{proof}

\subsection{Exact detection on the spectrum}
The disappearing directions admit useful equivalent descriptions. Define
\[
 \mathcal V=\int_ZB\dd\mu,
 \qquad \mathcal D_{\mathrm{int}}=\int_Z\DD\dd\mu=2I-\mathcal V.
\]
Since the integrands are nonnegative,
\begin{equation}\label{eq:integratedloss}
 \kk=\ker\mathcal V=\ker(2I-\mathcal D_{\mathrm{int}}).
\end{equation}
In particular, nonzero partial loss does not imply disappearance. For $\xi\in\kk$,
\[
 \int_{M_j}|J_j\nabla u_{j,\xi}|^2\dd\mu_j\longrightarrow2|\xi|^2,
\]
although $X_{j,\xi}\phi_{j,\ell}\to0$ for every fixed $\ell$.

On each complexified eigenspace, the commuting skew-Hermitian operators $D_z$, $z\in\z$, have simultaneous weights $\alpha\in\z^*$. Let $\mathcal W$ be the collection of all such weights over the full spectrum. Then
\begin{equation}\label{eq:weights}
 \kk=\bigcap_{\alpha\in\mathcal W}\ker\alpha,
 \qquad
 \Gamma=\{z\in\z:\alpha(z)\in2\pi\Z\text{ for all }\alpha\in\mathcal W\}
\end{equation}
is the kernel of the central additive-group action. The weights on the critical space itself are zero, so higher spectral levels are needed to detect visible central directions. The image of the central additive-group action  need not be closed in its compact closure; consequently a real infinitesimal direction is not automatically a primitive integral circle direction.

For a quantitative finite-window test, define
\[
 \mathcal H_t(\xi,\eta)=\sum_{\ell\geq0}e^{-t\nu_\ell}
 \ip{D_\xi\phi_\ell}{D_\eta\phi_\ell}_{L^2},\qquad t>0.
\]
Its kernel is $\kk$. If $\mathcal H_{t,R}$ is the sum over $\nu_\ell\leq R$, then
\begin{equation}\label{eq:heattail}
 \norm{\mathcal H_t-\mathcal H_{t,R}}_\op
 \leq C_nt^{-n-1}e^{-tR/2},\qquad0<t\leq1.
\end{equation}
Indeed, \eqref{eq:Denergy} bounds a summand by $C_n|\xi|^2\nu_\ell e^{-t\nu_\ell}$. The normalized heat-kernel bound gives $\sum e^{-s\nu_\ell}\leq C_ns^{-n}$; splitting off $e^{-tR/2}$ proves \eqref{eq:heattail}. If $\kappa_t>0$ is the least eigenvalue of $\mathcal H_t$ on $\kk^\perp$, a cutoff for which the right-hand side is below $\kappa_t/2$ gives
\[
 \ker\mathcal H_{t,R}=\kk.
\]
The cutoff depends on the visibility gap $\kappa_t$, not just on the original dimension.

\section{High multiplicity and almost rigidity}

\subsection{A single coadjoint image and noncollapsing}
Assume $q>n^2$. By Corollary~\ref{cor:centerexclusion}, the algebra is compact semisimple and both $G$ and $P$ have rank $2n$ almost everywhere. Let $S$ act as in Proposition~\ref{prop:globalaction}. Define
\[
 U:Z\longrightarrow\g^*,\qquad U(x)(\xi)=u_\xi(x).
\]
It is Lipschitz and equivariant for the right coadjoint action
\[
 (\nu\cdot s)(\xi)=\nu(\Ad_s\xi).
\]

\begin{lemma}\label{lem:oneorbit}
There is a compact coadjoint orbit $\OO\subset\g^*$ of real dimension $2n$ such that
\[
 U(Z)=\OO,\qquad U(x\cdot S)=\OO\quad\text{for every }x\in Z.
\]
\end{lemma}
\begin{proof}
For a smooth invariant function $F$ on $\g^*$, one has $P\,dF=0$ along the image. Since $\ker G=\ker P$ almost everywhere, the Sobolev chain rule gives
\[
 |\nabla(F\circ U)|^2=(dF)^TG(dF)=0.
\]
Thus $F\circ U$ is constant, first almost everywhere and then everywhere by continuity. Distinct compact orbits can be separated by smooth invariant functions: choose a smooth separating function and average it over $S$. Hence the entire image lies in one orbit. At a full-rank point its dimension is $\rank P=2n$. Equivariance and transitivity of the coadjoint action give surjectivity on each spatial orbit.
\end{proof}

\begin{proposition}\label{prop:noncolltrans}
The absolute Riemannian volumes converge to a positive number, and $S$ acts transitively on $Z$.
\end{proposition}
\begin{proof}
Give $\g^*$ an invariant Euclidean metric, and let $L$ be a Lipschitz constant for $U$. Every compact orbit $x\cdot S$ maps onto $\OO$, so
\begin{equation}\label{eq:orbitmeasure}
 \HH^{2n}(x\cdot S)\geq L^{-2n}\HH^{2n}(\OO)>0.
\end{equation}
Choose a radius larger than the common diameter bound. Each entire space is a closed ball of that radius. The Hausdorff-measure continuity theorem \cite[Theorem 1.3, version 3]{DPG} states that $\HH^{2n}$ is finite and continuous on this GH family of balls in $\RCD(1,2n)$ spaces. With Euclidean normalization of Hausdorff measure, it gives
\begin{equation}\label{eq:volumeconv}
 \HH^{2n}(Z)=\lim_j\Vol_{g_j}(M_j)<\infty.
\end{equation}
This application does not require a volume lower bound in advance. Equation~\eqref{eq:orbitmeasure} makes the limit strictly positive.

The lower bound in \eqref{eq:orbitmeasure} is the same for every orbit in this fixed limit. Since distinct orbits are disjoint compact measurable sets and the total measure is finite, there are only finitely many. Each is then open as well as closed. Connectedness gives a single orbit.
\end{proof}

\begin{corollary}\label{cor:homogeneous}
The limit is a smooth compact homogeneous Riemannian manifold $(Z,g_\infty)$ of real dimension $2n$. Its reference measure is normalized Riemannian volume and $\Ric_{g_\infty}\geq g_\infty$.
\end{corollary}
\begin{proof}
A transitive measure-preserving isometric action gives local metric measure homogeneity. Apply \cite[Theorem 1.2, version 2]{HN}. It yields a smooth Riemannian metric and a constant multiple of its volume measure. The positive finite $\HH^{2n}$ measure forces dimension $2n$, and probability normalization fixes the constant. On this unweighted smooth manifold the RCD bound is the usual Ricci bound.
\end{proof}

\subsection{Reconstruction of the K\"ahler structure}
The isometric action on the smooth limit is smooth by Myers--Steenrod \cite{MS}. The critical eigenfunctions are smooth by elliptic regularity. Write $Y_a$ for the fundamental fields.

\begin{proposition}\label{prop:kahlerreconstruction}
There is a smooth $S$-invariant integrable complex structure $J_\infty$ such that $g_\infty$ is K\"ahler and
\begin{equation}\label{eq:Jreconstruct}
 J_\infty\nabla u_a=Y_a.
\end{equation}
Its K\"ahler form is $U^*\omega_{\KKS}$. Moreover,
\[
 \Ric(\omega_\infty)=\omega_\infty,\qquad -\Delta_{\omega_\infty}u_a=u_a.
\]
\end{proposition}
\begin{proof}
The identities
\begin{equation}\label{eq:YP}
 \ip{Y_a}{\nabla u_b}=D_au_b=P_{ab}
\end{equation}
hold everywhere. Positivity and the complex-rank bound extend from a full-measure set by smoothness. The rank of $P$ is the dimension of the fixed coadjoint orbit and equals $2n$ everywhere. Lemma~\ref{lem:rank} therefore gives full rank for $G$ and equality of the kernels everywhere.

If $\sum t_a\nabla u_a=0$, then $Gt=0$ and hence $Pt=0$. By \eqref{eq:YP}, $\sum t_aY_a$ is orthogonal to all the gradients and is zero. Thus \eqref{eq:Jreconstruct} is well-defined. A local independent family of gradients shows that it is smooth. The skew symmetry of $P$ makes it skew-adjoint, and Lemma~\ref{lem:rank}(iii) gives
\[
 J_\infty^2=-I,\qquad
 g_\infty(J_\infty V,J_\infty W)=g_\infty(V,W).
\]
Equivariance makes this structure $S$-invariant.

For the right coadjoint convention, choose the KKS form characterized by
\[
 \omega_{\KKS}|_\nu(\xi^\#,\eta^\#)=\nu([\xi,\eta]).
\]
It is closed \cite{Rieffel}. Equivariance gives $U^*\omega_{\KKS}(Y_a,Y_b)=P_{ab}$. The form $\omega_\infty(V,W)=g_\infty(J_\infty V,W)$ has the same values, by \eqref{eq:YP}. The fundamental fields span every tangent space, so
\begin{equation}\label{eq:KKSpullback}
 \omega_\infty=U^*\omega_{\KKS},\qquad d\omega_\infty=0.
\end{equation}
The map $U$ is a local diffeomorphism because its differential is onto along the group action and source and target both have dimension $2n$.

Closedness alone would only prove an almost K\"ahler statement. To verify integrability, let $H_x\subset K_\nu$ be the stabilizers of $x$ and $\nu=U(x)$. Their Lie algebras agree. Identify $\nu$ with an element of $\g$ using the invariant inner product and choose a maximal torus containing it. This torus lies in $K_\nu^0=H_x^0$. The complexified tangent representation is the sum of the root spaces with nonzero value on $\nu$. Torus invariance preserves each real root plane. Compatibility with the fixed KKS form and positivity select one of the two invariant complex structures on that plane. These are the canonical choices for the coadjoint orbit. Equivalently, up to the common right-action sign, the $(1,0)$ roots are those with positive value on $\nu$. Their sums, when roots, have positive value, and brackets with zero-value isotropy roots preserve the choice. Thus the sum of isotropy and $(1,0)$ spaces is a complex subalgebra. The invariant $(1,0)$ distribution is involutive, proving integrability.

We now have a smooth K\"ahler metric with $\Ric\geq\omega_\infty$ and complex eigenvalue one for all $u_a$. Apply Lemma~\ref{lem:bochner} at equality. The nonnegative Ricci defect vanishes in every critical gradient direction. These gradients span the tangent space, so the defect is zero.
\end{proof}

\begin{corollary}\label{cor:highmodels}
If $q\geq n^2+3$, the limit is biholomorphically isometric to $(\CP^n,\omega_n)$. If $q>n^2$, the other possible limits are the normalized product $\CP^{n-1}\times\CP^1$ and, when $n=3$, the smooth quadric in $\CP^4$ with its canonical K\"ahler--Einstein metric.
\end{corollary}
\begin{proof}
Apply \cite[Theorem 1.6]{CWZ} after Proposition~\ref{prop:kahlerreconstruction}. The two remaining multiplicities are classified by \cite[Theorem 7.1]{CWZ}. The exact rigidity results are used only at this smooth K\"ahler stage.
\end{proof}

\subsection{The almost-rigidity theorems}
The preceding arguments prove a sequential statement: if $\lambda_{n^2+1}(\omega_j)\to1$, every subsequence has a further subsequence converging to a smooth homogeneous K\"ahler--Einstein limit, with convergence of absolute volumes. This immediately gives the noncollapsing assertion in Theorem~\ref{thm:noncollapseintro}. Otherwise one could arrange $\lambda_{n^2+1}\leq1+j^{-1}$ and $\Vol(M_j)<j^{-1}$, contradicting \eqref{eq:volumeconv}. Sharpness is established in Section~\ref{sec:examples}.

For the biholomorphic conclusion we use a separate theorem; it does not follow just from GH closeness.
\begin{lemma}\label{lem:biholomorphicgap}
Suppose $\Ric(\omega_j)\geq\omega_j$ and $(M_j,g_j)$ converges to $(\CP^n,g_n)$ in GH topology with convergence of absolute volumes. Then $M_j$ is biholomorphic to $\CP^n$ for all sufficiently large $j$.
\end{lemma}
\begin{proof}
Set $\widehat\omega_j=\omega_j/(n+1)$. Then
$\Ric(\widehat\omega_j)\geq(n+1)\widehat\omega_j$, and its volume tends to the volume of $\omega_n/(n+1)$. The almost-maximal-volume theorem in \cite[Theorem 1.3, proved in Liu's appendix]{Zhang} gives the conclusion. The fixed factor $n!$ between volume conventions changes only the dimensional gap constant.
\end{proof}

\begin{proof}[Proof of Theorem~\ref{thm:mainrigidity}]
If the assertion failed, there would be a sequence with $\Ric(\omega_j)\geq\omega_j$ and $\lambda_{n^2+3}(\omega_j)\leq1+j^{-1}$ for which either the biholomorphic conclusion fails or the GH distance is bounded below by a fixed positive number. Its full critical limit has $q\geq n^2+3$. By Corollary~\ref{cor:highmodels} and \eqref{eq:volumeconv}, it converges to normalized $\CP^n$ with absolute volume convergence. Lemma~\ref{lem:biholomorphicgap} gives the biholomorphic conclusion, and the GH distances tend to zero. This is a contradiction.
\end{proof}

\begin{corollary}\label{cor:averageddefect}
Under a sequence of hypotheses with $\lambda_{n^2+3}(\omega_j)\to1$,
\[
 \int_{M_j}\tr_{\omega_j}(\Ric(\omega_j)-\omega_j)\dd\mu_j\longrightarrow0.
\]
\end{corollary}
\begin{proof}
After the biholomorphic identifications, write $[\omega_j]=t_j[\omega_n]$, since $H^{1,1}(\CP^n,\R)$ is one-dimensional. The nonnegative Ricci defect gives $0<t_j\leq1$, and volume convergence gives $t_j\to1$. Cohomological integration gives the exact formula
\[
 \int\tr_{\omega_j}(\Ric(\omega_j)-\omega_j)\dd\mu_j=n(t_j^{-1}-1).
\]
This is an equality of cohomological integrals; it does not assert $\omega_j=t_j\omega_n$ as metrics.
\end{proof}

Neither the preceding proof nor the conclusion asserts smooth convergence of the original metrics, a curvature upper bound, or an explicit power law for $\delta(n,\epsilon)$. The compactness argument and the separate volume-gap theorem are sufficient for the stated almost rigidity.

\section{Sharp examples in every dimension}\label{sec:examples}

\subsection{A smooth collapsing sphere}
\begin{proposition}\label{prop:sphere}
There are smooth K\"ahler metrics $(S^2,g_\beta,J_\beta,\omega_\beta)$, $0<\beta<1$, such that
\[
 K_{g_\beta}\geq1,\qquad\operatorname{Area}(g_\beta)=4\pi\beta,
 \qquad 2\leq\nu_1(g_\beta)\leq2+6\beta.
\]
They converge in normalized measured GH topology to
\[
 I=\left([0,\pi],dt^2,\tfrac12\sin t\dd t\right).
\]
If the real eigenvalues are numbered starting with $\nu_0=0$, then for each fixed $\ell$,
\begin{equation}\label{eq:spherespectrum}
 \nu_\ell(g_\beta)\longrightarrow\ell(\ell+1).
\end{equation}
In particular, exactly one nonzero mode tends to the critical real eigenvalue two.
\end{proposition}
\begin{proof}
For $\epsilon>0$ and $-1\leq x\leq1$, set
\[
 h_\epsilon(x)=2\epsilon\,
 \frac{\cosh(1/\epsilon)-\cosh(x/\epsilon)}{\sinh(1/\epsilon)}.
\]
Then $h_\epsilon(\pm1)=0$, $h_\epsilon'(-1)=2$, $h_\epsilon'(1)=-2$, $h_\epsilon''<0$, and $0\leq h_\epsilon\leq2\epsilon$. Define
\begin{equation}\label{eq:spheremetric}
 F_\beta(x)=\beta(1-x^2)+(1-\beta)h_{\beta^2}(x),\qquad
 g_\beta=\beta\left(\frac{dx^2}{F_\beta(x)}+F_\beta(x)\,d\theta^2\right),
\end{equation}
where $\theta$ has period $2\pi$.

The function $F_\beta$ is analytic, positive in $(-1,1)$, vanishes at the endpoints, and has slopes $2,-2$. To verify smoothness, put $s=1-x$ near the upper endpoint. Then $F_\beta(1-s)=2s+O(s^2)$, and the radial distance has the form
\[
 \rho=\sqrt\beta\int_0^sF_\beta(1-u)^{-1/2}\dd u
      =\sqrt{2\beta s}\,b_\beta(s),\qquad b_\beta(0)=1,
\]
with $b_\beta$ analytic. Thus $s=\rho^2c_\beta(\rho^2)$ and the angular radius is $\rho a_\beta(\rho^2)$, with $a_\beta(0)=1$. This is a smooth polar metric. The lower endpoint is identical. Direct computation gives
\begin{equation}\label{eq:spherecurvature}
 dA_\beta=\beta\dd x\dd\theta,
 \qquad K_{g_\beta}=-\frac{F_\beta''}{2\beta}\geq1.
\end{equation}
Every oriented surface metric defines a compatible complex structure. Its K\"ahler Ricci form is $K_{g_\beta}\omega_\beta$, so it satisfies the desired Ricci inequality.

The odd function $x$ has mean zero and Rayleigh quotient
\[
 \frac{\int|\nabla x|^2\dd A_\beta}{\int x^2\dd A_\beta}
 =\frac3{2\beta}\int_{-1}^1F_\beta(x)\dd x\leq2+6\beta.
\]
The lower bound follows from Lemma~\ref{lem:bochner} in complex dimension one.

Put $A_\beta=F_\beta/\beta$. Then
\begin{equation}\label{eq:Acomparison}
 1-x^2\leq A_\beta(x)\leq1-x^2+2\beta.
\end{equation}
The meridian coordinate
\[
 t_\beta(x)=\int_{-1}^xA_\beta(s)^{-1/2}\dd s
\]
converges uniformly to $t_0(x)=\arccos(-x)$ by dominated convergence. Every circular fiber has intrinsic diameter at most $\pi\beta\sqrt{1+2\beta}$. Distances between points lie between their meridian-distance difference and that difference plus this fiber diameter. Thus $\Phi_\beta(x,\theta)=t_0(x)$ is a surjective GH approximation. Since
\[
 \mu_\beta=\frac1{4\pi}\dd x\dd\theta,
 \qquad (\Phi_\beta)_\#\mu_\beta=\tfrac12\sin t\dd t,
\]
the convergence is measured.

For the spectral assertion, an angular mode $v(x)e^{im\theta}$ has real Rayleigh quotient
\begin{equation}\label{eq:angularquotient}
 \frac{\displaystyle\int_{-1}^1
 \left(A_\beta|v'|^2+\frac{m^2}{\beta^2A_\beta}|v|^2\right)\dd x}
 {\displaystyle\int_{-1}^1|v|^2\dd x}.
\end{equation}
When $m\neq0$, this is at least $m^2/[\beta^2(1+2\beta)]$ and tends to infinity. The radial operator is
\[
 -\frac d{dx}\left(A_\beta\frac d{dx}\right)
 \quad\text{on }L^2([-1,1],dx/2).
\]
For every fixed $\beta$, $A_\beta/(1-x^2)$ extends at both endpoints to finite positive values. Thus its closed energy domain agrees with the domain of the Legendre form. The latter has eigenfunctions $P_\ell(x)$ and eigenvalues $\ell(\ell+1)$. Indeed,
\[
 -\bigl((1-x^2)P_\ell'\bigr)'=\ell(\ell+1)P_\ell,
\]
and the endpoint flux vanishes. Polynomial density proves completeness; the operator is the natural self-adjoint operator of the energy form, not an imposed Dirichlet problem.

The lower comparison in \eqref{eq:Acomparison} gives radial eigenvalues at least $\ell(\ell+1)$. On the finite-dimensional space spanned by $P_0,\ldots,P_\ell$, the upper comparison gives the reverse bound $\ell(\ell+1)+C_\ell\beta$ by min--max. All nonradial modes leave every bounded window by \eqref{eq:angularquotient}. This proves \eqref{eq:spherespectrum}.
\end{proof}

\subsection{Sharpness of the dimension profile}
On normalized $\CP^k$, the real critical multiplicity is
\[
 q_k=(k+1)^2-1=k^2+2k.
\]
One explicit realization is $f_A([z])=z^*Az/(z^*z)$ for traceless Hermitian matrices $A$; these exhaust the eigenvalue-one functions \cite[Lemma 6.1]{CWZ}. For $k=0$, the space is a point and $q_0=0$. The next eigenvalue is strictly greater than the critical one.

For even $r=2k$, $0\leq k\leq n$, take
\begin{equation}\label{eq:evensharp}
 M_\beta=\CP^k\times\CP^{n-k},\qquad
 \Omega_\beta=\omega_k+\beta\omega_{n-k},
\end{equation}
omitting point factors. Its Ricci form is $\omega_k+\omega_{n-k}\geq\Omega_\beta$. Projection to the first factor gives normalized measured convergence to $\CP^k$. The second factor shrinks to a point, and all its nonzero complex eigenvalues tend to infinity. Consequently the critical multiplicity is exactly $q_k=Q(2k)$.

For odd $r=2k+1$, $0\leq k\leq n-1$, let $l=n-k-1$ and take
\begin{equation}\label{eq:oddsharp}
 M_\beta=\CP^k\times(S^2,J_\beta)\times\CP^l,
 \qquad\Omega_\beta=\omega_k+\omega_\beta+\beta\omega_l.
\end{equation}
It satisfies the same Ricci inequality and converges to $\CP^k\times I$. The product Laplacian is the sum of the factor Laplacians. Since every nonconstant bounded-window factor mode has limiting real eigenvalue at least two, mixed nonconstant modes have limiting eigenvalue at least four. Hence
\[
 \dim\ker(\LL_{\CP^k\times I}-2)=q_k+1=(k+1)^2=Q(2k+1).
\]
The next eigenvalue remains strictly above two. The essential dimensions are the displayed dimensions: on the interior of $I$ the density is smooth and positive, and the endpoints have zero measure.

Writing $v_k=\Vol(\CP^k,g_k)$ and $v_0=1$, the absolute volumes in \eqref{eq:evensharp} and \eqref{eq:oddsharp} are respectively
\[
 v_kv_{n-k}\beta^{n-k},\qquad
 4\pi v_kv_l\beta^{l+1}.
\]
They tend to zero whenever $r<2n$. At $r=0$ the whole manifold shrinks to a point; all nonzero eigenvalues diverge. At $r=2n$, the constant $\CP^n$ sequence is the equality model. This proves all the sharpness claims in Theorem~\ref{thm:dimensionintro}.

Taking $k=n-1$ in \eqref{eq:oddsharp} proves Theorem~\ref{thm:noncollapseintro}. Taking the unscaled normalized product $\CP^{n-1}\times\CP^1$ gives $n^2+2$ exact critical modes and proves sharpness of the index in Theorem~\ref{thm:mainrigidity}.

For completeness, $Q(2k)=(k+1)^2-1$ and $Q(2k+1)=(k+1)^2$. Inverting these inequalities gives \eqref{eq:inverseprofile}. The examples at the minimal allowed dimension have at least the requested $p$ critical modes, so the resulting maximal dimension loss is sharp in precisely that sense.

\subsection{Sharpness of disappearing symmetry}
Denote the sphere in Proposition~\ref{prop:sphere} by $C_\beta$. For $n\leq r\leq2n$, take
\[
 \CP^{r-n}\times C_\beta^{\,2n-r}.
\]
The original complex dimension is $n$ and the limiting real dimension is $r$. Each collapsing sphere contributes one central critical function. Its Hamiltonian derivative annihilates every fixed bounded spectral window for sufficiently small $\beta$: such windows contain only radial modes in that factor. The projective factor retains its faithful semisimple action. Thus
\[
 \dim\kk=2n-r.
\]
For $0\leq r\leq n$, take
\[
 C_\beta^{\,r}\times(\CP^{n-r},\beta g_{n-r}).
\]
The limit is $I^r$, with $r$ critical modes and all their actions invisible. Hence $\dim\kk=r$. The product spectrum again excludes extra critical mixed modes. This proves sharpness in Theorem~\ref{thm:lossintro} for every dimension.

\subsection{A model quantitative estimate}
There is a useful estimate on the comparison model which does not require a quantitative version of Theorem~\ref{thm:mainrigidity}.
\begin{proposition}\label{prop:modeltail}
Let $f_1,\ldots,f_k$, $1\leq k\leq n$, be real $L^2$-orthonormal commuting critical functions on normalized $\CP^n$, and let $h_{\min}$ be the smallest eigenvalue of their gradient Gram matrix. Then
\begin{equation}\label{eq:modeltail}
 \mu_n\{h_{\min}\leq s\}\leq\frac{n(n+1)}{2(n+2)}s,
 \qquad0<s\leq1.
\end{equation}
The exponent one is optimal uniformly over these subspaces.
\end{proposition}
\begin{proof}
The corresponding Hermitian matrices commute, so they may be simultaneously diagonalized. Put $p_i=|z_i|^2/\sum|z_j|^2$. There is a real $k\times(n+1)$ matrix $Q$ with $QQ^T=I$ and $Q\mathbf1=0$ such that
\[
 f_\alpha=\sqrt{(n+1)(n+2)}\sum_iQ_{\alpha i}p_i,
 \qquad
 G=2(n+2)Q(\operatorname{diag}p-pp^T)Q^T.
\]
For $v\perp\mathbf1$, the weighted variance satisfies
\[
 v^T(\operatorname{diag}p-pp^T)v
 =\sum_i p_i(v_i-\bar v_p)^2
 \geq(\min_i p_i)|v|^2.
\]
Thus $h_{\min}\geq2(n+2)\min_i p_i$. Under normalized Fubini--Study measure, $p$ is uniformly distributed on the standard simplex. This follows by normalizing the squared moduli of independent standard complex Gaussian variables. Hence
\[
 \mu_n\{p_i\leq t\}=1-(1-t)^n\leq nt.
\]
The union bound proves \eqref{eq:modeltail}.

To see optimality, include $(n,-1,\ldots,-1)/\sqrt{n(n+1)}$ as a row of $Q$. Its diagonal Gram entry is
\[
 \frac{2(n+2)(n+1)}n p_0(1-p_0).
\]
Since $h_{\min}$ is bounded above by this entry, the set $p_0\leq c_ns$ gives a lower bound $c_n's$ for small $s$. No uniform exponent larger than one is possible.
\end{proof}

If $G_M$ is a comparison Gram matrix on another probability space and $\pi$ is a probability coupling with the model, put
\[
 \varepsilon_G^2=\int\norm{G_M(x)-G_*(y)}_\HS^2\dd\pi(x,y).
\]
The minimum-eigenvalue perturbation inequality and Chebyshev's inequality give, for every $\tau>0$,
\[
 \mu\{\lambda_{\min}(G_M)\leq s\}
 \leq C_n(s+\tau)+\frac{\varepsilon_G^2}{\tau^2}.
\]
Taking $\tau=\varepsilon_G^{2/3}$ gives $C_n(s+\varepsilon_G^{2/3})$ when $\varepsilon_G\leq1$. This is a quantitative transfer estimate in the stated comparison error. It is not an estimate of that error by a power of spectral pinching.

\section{Anticanonical sections from the spectrum}

\subsection{A volume-independent holomorphic projection}
The individual complex gradients need not admit a uniformly controlled projection to holomorphic vector fields under a Ricci lower bound alone. Their highest exterior powers take values in a positive line bundle, and satisfy a different estimate.

Equip $K_M^{-1}$ and its tensor powers with the metrics induced by $g$. The curvature of $K_M^{-1}$ is $\Ric(\omega)$.
\begin{lemma}\label{lem:holprojection}
Suppose $\Ric(\omega)\geq\omega$. For every integer $m\geq0$ and every smooth section $s$ of $K_M^{-m}$, its holomorphic orthogonal projection satisfies
\begin{equation}\label{eq:holprojection}
 \norm{s-P^{\mathrm{hol}}s}_2^2
 \leq\frac1{m+1}\norm{\dbar s}_2^2.
\end{equation}
The estimate holds for either ordinary or normalized volume.
\end{lemma}
\begin{proof}
There is a canonical isometric holomorphic identification
\[
 \Lambda^{0,q}T^*M\otimes K_M^{-m}
 \simeq\Lambda^{n,q}T^*M\otimes K_M^{-(m+1)}.
\]
Locally, it inserts a top holomorphic form and its dual; their norms cancel. It commutes with $\dbar$. The coefficient line bundle on the right has curvature
\[
 i\Theta(K_M^{-(m+1)})=(m+1)\Ric(\omega)\geq(m+1)\omega.
\]
The H\"ormander--Nakano estimate for closed $(n,1)$-forms therefore solves $\dbar v=\dbar s$ with $\norm v_2^2\leq(m+1)^{-1}\norm{\dbar s}_2^2$; see \cite[Theorem 1.1]{Raufi}. The minimal solution is $s-P^{\mathrm{hol}}s$. Multiplication of the measure by a constant changes both sides equally.
\end{proof}

\begin{proof}[Proof of Theorem~\ref{thm:recoveryintro}]
The spectral bounds give $\norm{W_a}_\infty\leq C_n$ and $\norm{\dbar W_a}_2\leq C\sqrt\eta$, by Lemmas~\ref{lem:bochner} and~\ref{lem:lowfrequency}. For $s_I=W_{i_1}\wedge\cdots\wedge W_{i_n}$,
\[
 \dbar s_I=\sum_{l=1}^n
 W_{i_1}\wedge\cdots\wedge\dbar W_{i_l}\wedge\cdots\wedge W_{i_n}.
\]
Thus $\norm{\dbar s_I}_2\leq C_n\sqrt\eta$. Lemma~\ref{lem:holprojection} with $m=1$ proves \eqref{eq:recoveryintro}. Also,
\begin{equation}\label{eq:sectionnormbound}
 \sum_I\norm{\sigma_I}_2^2\leq\sum_I\norm{s_I}_2^2\leq C(n,q),
\end{equation}
since orthogonal projection is contractive.

For a sequence with limiting complex Gram matrix $H=G+iP$, Cauchy--Binet gives
\begin{equation}\label{eq:wedgeGram}
 S_j:=\sum_I|s_{j,I}|^2=e_n(H_j),
\end{equation}
where $e_n$ is the $n$th elementary symmetric polynomial. The matrices are uniformly bounded and converge strongly in $L^2$, so $S_j\to e_n(H)$ strongly in $L^2$. If $\rank_\C H=n$ almost everywhere, this limiting function is positive almost everywhere. Its integral is positive, and \eqref{eq:recoveryintro} shows that not all $\sigma_{j,I}$ can be zero for large $j$.

To identify the projective evaluations, choose a local unitary frame of $K_M^{-1}$ and write the sections as vectors $s=(s_I)$ and $\sigma=(\sigma_I)$. The associated projectors are $\Pi_s=ss^*/|s|^2$ and $\Pi_\sigma=\sigma\sigma^*/|\sigma|^2$. For every $\tau>0$,
\begin{equation}\label{eq:projectivecomparison}
 \int_{\{S_j\geq\tau\}}\norm{\Pi_{s_j}-\Pi_{\sigma_j}}_\HS^2\dd\mu_j
 \leq\frac{C(n,q)\eta_j}{\tau}.
\end{equation}
This follows from the elementary inequality
$\norm{\Pi_s-\Pi_\sigma}_\HS^2\leq C|s-\sigma|^2/|s|^2$.
The common zero set of the nontrivial holomorphic system is a proper analytic subset and has volume zero; arbitrary values there do not affect the estimate.

Moreover, with the conjugate-linear-first convention,
\begin{equation}\label{eq:PluckerGram}
 (\Pi_{s_j})_{IJ}=\frac{\det(H_j)_{J,I}}{e_n(H_j)}.
\end{equation}
First pass to the limit on $e_n(H)>\tau$ and then let $\tau\downarrow0$. Equations~\eqref{eq:projectivecomparison} and~\eqref{eq:PluckerGram} prove recovery in measure. This is a statement about projective evaluations on the measured limit, not an identification of a complex manifold structure on an odd-dimensional space.
\end{proof}

Corollary~\ref{cor:dimensionequality} shows that the nonvanishing hypothesis applies at the extremal collapsing multiplicity $q=n^2$. It also applies in the high-multiplicity noncollapsed case.

\begin{corollary}\label{cor:measurepartialC0}
Let $\mathcal B_j=\sum_\alpha|\psi_{j,\alpha}|^2$ be the Bergman density of $H^0(M_j,-K_{M_j})$ for an $L^2(\mu_j)$-orthonormal basis. Under the rank hypothesis of Theorem~\ref{thm:recoveryintro},
\begin{equation}\label{eq:measurepartialC0}
 \mu_j\{\mathcal B_j<c_n\tau\}
 \leq\mu_j\{S_j<\tau\}+\frac{C_n\eta_j}{\tau},
\end{equation}
where the constants may also depend on the fixed cluster size. Consequently
\[
 \lim_{\tau\downarrow0}\limsup_j\mu_j\{\mathcal B_j<c_n\tau\}=0.
\]
\end{corollary}
\begin{proof}
By \eqref{eq:sectionnormbound} and the evaluation inequality,
$\sum_I|\sigma_{j,I}|^2\leq C_n\mathcal B_j$.
Let $E_j=\sum_I|\sigma_{j,I}-s_{j,I}|^2$. On $S_j\geq\tau$ and $E_j\leq\tau/4$, one has $\sum_I|\sigma_{j,I}|^2\geq\tau/4$. Chebyshev's inequality and \eqref{eq:recoveryintro} give \eqref{eq:measurepartialC0}. Since $S_j\to e_n(H)>0$ almost everywhere in the measured sense, the iterated limit follows.
\end{proof}

This estimate uses the first anticanonical power and no volume lower bound. It is not a pointwise partial $C^0$ estimate. Also, orthogonal projection is not a homomorphism of algebras. Although the raw exterior products satisfy the Pl\"ucker relations, the projected sections need not satisfy them exactly. Their projective image is therefore not asserted to lie in a Grassmannian.

\subsection{Divisors and their degrees}
Positivity of the metric on $-K_M$ makes it ample by the Kodaira embedding theorem; see \cite[Chapter VII]{Demailly}. Thus every manifold in this paper is a smooth Fano manifold. Boundedness of smooth Fano varieties in fixed dimension, as included in \cite{Birkar}, gives an integer $m_0(n)$ and a constant $V_n$ such that $-m_0K_M$ is very ample and $(-K_M)^n\leq V_n$.

For any nonzero spectral section, write
\[
 \operatorname{div}(\sigma)=\sum_\alpha b_\alpha D_\alpha,
 \qquad b_\alpha\in\Z_{>0}.
\]
With $A=-m_0K_M$, its degree is
\[
 \operatorname{div}(\sigma)\cdot A^{n-1}
 =m_0^{n-1}(-K_M)^n\leq m_0^{n-1}V_n.
\]
Each irreducible divisor has positive integral degree in this embedding. Hence
\begin{equation}\label{eq:divisorbound}
 \sum_\alpha b_\alpha\leq m_0^{n-1}V_n.
\end{equation}
The same bound applies to the common fixed divisor of the system. This controls exact divisor multiplicities, but does not imply that a moving system cannot acquire a fixed divisor in a limit.

For clarity, let $\sigma_1,\ldots,\sigma_N$ be any nonzero system in a line bundle $L$, and define its common fixed divisor by
\[
 \mathcal F=\sum_D\left(\min_i\operatorname{ord}_D\sigma_i\right)D.
\]
Writing $\sigma_i=s_{\mathcal F}\widetilde\sigma_i$, the residual sections lie in $L(-\mathcal F)$ and have no common divisorial component. Put $dd^c=(i/2\pi)\partial\dbar$ when acting on logarithms of squared norms. Poincar\'e--Lelong \cite[Chapter V, Section 13]{Demailly} gives the exact current identity
\begin{equation}\label{eq:fixedmoving}
 c_1(L,h)+dd^c\log\sum_i|\sigma_i|_h^2
 =[\mathcal F]+T^{\mathrm{mov}},\qquad T^{\mathrm{mov}}\geq0.
\end{equation}
The remaining system may still have higher-codimension base points. The following section shows that even basepoint-free spectral systems can develop end defects in their moving part.

\section{End defects of the moving spectral system}\label{sec:ends}

In this section the complex dimension is two. We prove Theorem~\ref{thm:counterintro}. The perturbation is carried out on $X=\CP^1\times\CP^1$ with its product complex structure, after identifying the second factor in Proposition~\ref{prop:sphere} with $\CP^1$. The estimates are diagonal in the collapsing parameter. No bound uniform in that parameter is required for the perturbation radius.

\subsection{The unperturbed spectral system}
Let $\omega_B$ be the curvature-one round metric on the first $\CP^1$, so that $\int\omega_B=4\pi$ and $\Ric(\omega_B)=\omega_B$. Write $a_1,a_2,a_3$ for its coordinate functions on the unit sphere. Then
\begin{equation}\label{eq:basecoords}
 -\Delta_Ba_i=a_i,\qquad \sum_i a_i^2=1,\qquad
 \int_Ba_ia_j\dd\mu_B=\frac{\delta_{ij}}3.
\end{equation}
Set $\Omega_\beta=\omega_B+\omega_\beta$ and use the notation $F_\beta$ of \eqref{eq:spheremetric}.

For small fixed $\beta$, the first four complex spectral directions are
\[
 a_1,\ a_2,\ a_3,\ h_\beta(x),
\]
where $h_\beta$ is the first nonconstant radial eigenfunction of the second sphere. Its eigenvalue lies in $(1,1+3\beta]$; strictness follows from $K_{g_\beta}>1$. The next radial eigenvalue is at least three, all angular modes are large, and the first mixed product mode is at least two. Thus
\begin{equation}\label{eq:surfacegap}
 \lambda_5(\Omega_\beta)\geq2
\end{equation}
for sufficiently small $\beta$.

Choose $h_\beta$ increasing. The one-dimensional first-eigenfunction nodal theorem and its differential equation give
\[
 A_\beta h_\beta'=-\nu_\beta\int_{-1}^x h_\beta(s)\dd s>0
 \quad(-1<x<1),
\]
where $\nu_\beta$ is its real eigenvalue. Introduce the holomorphic coordinate
\begin{equation}\label{eq:fibercoordinate}
 w=\exp(s(x)+i\theta),\qquad
 s(x)=\int_0^x\frac{du}{F_\beta(u)}.
\end{equation}
The poles are $w=0,\infty$. Set $v=w\partial_w$. Then
\[
 v=\tfrac12(F_\beta\partial_x-i\partial_\theta),\qquad
 |v|^2=\tfrac12\beta F_\beta,
 \qquad \nabla^{1,0}h_\beta=\frac{h_\beta'}\beta v.
\]
Rotation invariance of the holomorphic projection implies
\begin{equation}\label{eq:radialprojection}
 P^{\mathrm{hol}}_\beta\nabla^{1,0}h_\beta=\gamma_\beta v,
 \qquad
 \gamma_\beta=\frac1\beta
 \frac{\int_{-1}^1h_\beta'F_\beta\dd x}{\int_{-1}^1F_\beta\dd x}>0.
\end{equation}
The only rotation-invariant holomorphic vector fields are multiples of $v$.

Let $Z_a=\nabla_B^{1,0}a_a$. On $X$, all horizontal wedges $Z_a\wedge Z_b$ vanish. The three nonzero projected determinants are
\begin{equation}\label{eq:unperturbedsections}
 \sigma_{a4}(0)=\gamma_\beta Z_a\otimes v\qquad(a=1,2,3),
\end{equation}
up to fixed normalization of the real basis. The $Z_a$ have no common zero; $v$ has a simple zero at both poles. Thus the unperturbed system has fixed divisor
\[
 D_-+D_+,\qquad D_-=B\times\{0\},\quad D_+=B\times\{\infty\}.
\]
After removing it, the system is the pullback of the anticanonical system of $B$. The perturbation below changes this conclusion at every nonzero parameter.

\subsection{An angular mode with nonzero holomorphic end values}
Choose the lowest eigenfunction in angular frequency one on the second sphere:
\begin{equation}\label{eq:angularmode}
 \varphi(x,\theta)=b(x)\cos\theta,\qquad b(x)>0\ (-1<x<1),
 \qquad-\Delta_\beta\varphi=\Lambda\varphi.
\end{equation}
It exists by separation of variables and the positivity of the first Sturm--Liouville eigenfunction in that mode. It is smooth at the poles, and $\Lambda>1$ for small $\beta$. The radial equation is
\[
 -\frac1{2\beta}(F_\beta b')'+\frac{1}{2\beta F_\beta}b=\Lambda b.
\]
This high-frequency mode is used only as a perturbing potential; it is not added to the critical cluster.

\begin{lemma}\label{lem:angularprojection}
The holomorphic projection of $\nabla^{1,0}\varphi$ is
\[
 \widehat Z=c_0\partial_w+c_2w^2\partial_w,
\]
with $c_0>0$ and $c_2<0$ in the coordinate \eqref{eq:fibercoordinate}. In particular it is nonzero at both $w=0$ and $w=\infty$.
\end{lemma}
\begin{proof}
The space of holomorphic vector fields on $\CP^1$ is spanned by $\partial_w,w\partial_w,w^2\partial_w$. Rotational weights imply that the middle coefficient is zero. Direct differentiation gives
\[
 \nabla^{1,0}\varphi
 =\frac1\beta\left(b'\cos\theta-i\frac b{F_\beta}\sin\theta\right)v.
\]
With normalized measure $dx\,d\theta/(4\pi)$, integration in $\theta$ and then by parts in $x$ yield
\begin{align}\label{eq:angularcoefficients}
 c_0\norm{\partial_w}_2^2
 &=\frac18\int_{-1}^1 b e^{-s}(2-F_\beta')\dd x>0,\notag\\
 c_2\norm{w^2\partial_w}_2^2
 &=-\frac18\int_{-1}^1 b e^s(2+F_\beta')\dd x<0.
\end{align}
The boundary terms are zero: smoothness of the angular mode gives square-root behavior of $b$ in the polar moment coordinate, while $F_\beta$ vanishes linearly. Strict concavity gives $-2<F_\beta'<2$ in the interior. Finally, $w^2\partial_w=-\partial_{1/w}$ at the upper pole, proving the last assertion.
\end{proof}

\subsection{Variation of the actual spectral projections}
For fixed $\beta$, consider
\begin{equation}\label{eq:mixperturbation}
 \Omega_{\beta,t}=\Omega_\beta+t\,i\partial\dbar(a_1\varphi).
\end{equation}
It is K\"ahler for all sufficiently small real $t$. The three-dimensional eigenvalue-one space is isolated from the radial fourth mode. Analytic spectral projection gives an analytic frame $u_a(t)$ of its continuation, with $u_a(0)=a_a$; see \cite[Chapter VII]{Kato}. We choose its first derivative perpendicular to the original eigenspace. One can obtain this frame by identifying the varying $L^2$ spaces through multiplication by the square root of the volume density and applying the Riesz projection; elliptic regularity gives analytic dependence in every fixed smooth norm. The fourth spectral direction is continued separately.

The determinant system depends only on the four-dimensional subspace, not on the chosen basis. Calculations may therefore be made in this convenient unnormalized frame. A final analytic orthonormalization gives the actual normalized spectral system used for its Fubini--Study current.

\begin{lemma}\label{lem:spectralvariation}
Let $\sigma_{ab}(t)$ be the holomorphic projection of
$\nabla_t^{1,0}u_a(t)\wedge\nabla_t^{1,0}u_b(t)$, for $1\leq a<b\leq3$. Then
\begin{equation}\label{eq:sectionvariation}
 \dot\sigma_{ab}(0)
 =-\kappa_\Lambda(\delta_{1b}Z_a-\delta_{1a}Z_b)\otimes\widehat Z,
 \qquad
 \kappa_\Lambda=\frac{\Lambda+1}{2(\Lambda-1)(\Lambda+2)}>0.
\end{equation}
In particular,
\[
 \dot\sigma_{12}(0)=\kappa_\Lambda Z_2\otimes\widehat Z,
 \qquad
 \dot\sigma_{13}(0)=\kappa_\Lambda Z_3\otimes\widehat Z.
\]
\end{lemma}
\begin{proof}
Put $L_t=-\Delta_{\Omega_{\beta,t}}$. Since $(a_a)_{z\bar z}=-a_a(g_B)_{z\bar z}$, variation of the inverse K\"ahler metric gives
\[
 \dot L_0a_a=a_1a_a\varphi.
\]
Its projection onto the original eigenspace is zero, because $\varphi$ has zero angular mean. Thus
\[
 (L_0-1)\dot u_a=-a_1a_a\varphi.
\]
The identity $a_1a_a=\delta_{1a}/3+(a_1a_a-\delta_{1a}/3)$ separates a constant mode from a degree-two spherical harmonic of complex eigenvalue three. Solving on the orthogonal complement gives
\begin{equation}\label{eq:eigenvariation}
 \dot u_a=-\varphi\left(
 \frac{a_1a_a}{\Lambda+2}
 +\frac{\delta_{1a}}{(\Lambda-1)(\Lambda+2)}\right).
\end{equation}

At $t=0$ the horizontal wedge is identically zero. Therefore differentiating the holomorphic projector contributes no term. The derivative of the inverse metric also contributes no term to that wedge: in product holomorphic coordinates,
\[
 \nabla^{1,0}f\wedge\nabla^{1,0}h
 =\frac{f_{\bar z}h_{\bar w}-f_{\bar w}h_{\bar z}}{\det g}
 \partial_z\wedge\partial_w,
\]
and its numerator is zero at $t=0$ for the horizontal functions. Consequently \eqref{eq:eigenvariation} gives
\begin{align}\label{eq:rawwedgevariation}
 \dot s_{ab}(0)=-\left[
 \frac{a_1(a_bZ_a-a_aZ_b)}{\Lambda+2}
 +\frac{\delta_{1b}Z_a-\delta_{1a}Z_b}{(\Lambda-1)(\Lambda+2)}
 \right]\otimes\nabla^{1,0}\varphi.
\end{align}

We spell out the required projection on the round sphere. In its standard coordinate,
\[
 a_1=\frac{z+\bar z}{1+|z|^2},\quad
 a_2=\frac{-i(z-\bar z)}{1+|z|^2},\quad
 a_3=\frac{1-|z|^2}{1+|z|^2},
\]
\[
 Z_1=\tfrac12(1-z^2)\partial_z,\quad
 Z_2=\tfrac i2(1+z^2)\partial_z,\quad
 Z_3=-z\partial_z.
\]
They satisfy
\[
 a_bZ_a-a_aZ_b=i\varepsilon_{abc}Z_c,
\]
\[
 \ip{Z_e}{Z_c}
 =\tfrac12(\delta_{ec}-a_ea_c+i\varepsilon_{ecd}a_d),
 \qquad\ip{Z_e}{Z_c}_{L^2}=\tfrac13\delta_{ec}.
\]
Equation~\eqref{eq:basecoords} and vanishing of odd spherical moments now give
\[
 P_B^{\mathrm{hol}}(a_dZ_c)=\frac i2\sum_e\varepsilon_{ecd}Z_e,
\]
and hence
\begin{equation}\label{eq:baseprojection}
 P_B^{\mathrm{hol}}\bigl(a_1(a_bZ_a-a_aZ_b)\bigr)
 =\tfrac12(\delta_{1b}Z_a-\delta_{1a}Z_b).
\end{equation}
The holomorphic projection on the unperturbed product factors as the tensor product of the two projections. Substituting \eqref{eq:baseprojection} and Lemma~\ref{lem:angularprojection} into \eqref{eq:rawwedgevariation} proves \eqref{eq:sectionvariation}.
\end{proof}

\begin{proposition}\label{prop:basepointfree}
For every sufficiently small fixed $\beta$, the six projected determinants form a basepoint-free system in $H^0(X,\OO(2,2))$ for all sufficiently small nonzero $t$. Its holomorphic evaluation map has complex two-dimensional image and pulls back $\OO(1)$ to $\OO(2,2)$.
\end{proposition}
\begin{proof}
The fields $Z_2,Z_3$ have no common zero. By Lemma~\ref{lem:angularprojection}, $\widehat Z$ is nonzero at both poles. Therefore the two derivatives in \eqref{eq:sectionvariation} have no common zero on a fixed neighborhood $U$ of $D_-\cup D_+$. After shrinking $U$, compactness gives a uniform positive lower bound for the sum of their squared norms. Since the two sections themselves vanish identically at $t=0$, their expansions are $t\dot\sigma(0)+O_\beta(t^2)$ in smooth norm. Thus they have no common zero on $U$ for small nonzero $t$.

On the compact complement of $U$, the three sections \eqref{eq:unperturbedsections} have no common zero, and their continuations retain this property. This proves basepoint freeness. A basepoint-free system generating $L$ satisfies $\Psi^*\OO(1)=L$. Since $\OO(2,2)$ is ample, a positive-dimensional fiber would contain a curve on which $L$ has degree zero, which is impossible. Hence the map is finite onto its image, and the image is two-dimensional.
\end{proof}

\subsection{Concentration and the Ricci-preserving diagonal}
Use an orthonormal analytic frame of the first four real spectral directions. Orthogonal changes of this frame act orthogonally on the six exterior products, so the sum of squared norms is intrinsic. Let $\Psi_{\beta,t}:X\to\CP^5$ be the associated map and set
\[
 T_{\beta,t}=\Psi_{\beta,t}^*\omega_{\FS},\qquad
 \int_{\CP^1}\omega_{\FS}=1.
\]
For fixed $\beta$, the holomorphic coefficients converge as $t\to0$ to the system \eqref{eq:unperturbedsections}. Near an end, let $\zeta$ be its holomorphic normal coordinate. The first-order terms in Lemma~\ref{lem:spectralvariation} are independent of the leading cross-section vector: the former have a nonzero component among the first three wedge coordinates, whereas the latter lies among the last three. It follows, after a finite covering of the end divisor, that
\begin{equation}\label{eq:potentialbound}
 c_\beta(|\zeta|^2+t^2)
 \leq\sum_{a<b}|f_{ab,t}|^2
 \leq C_\beta(|\zeta|^2+t^2)
\end{equation}
in local holomorphic line-bundle frames. For $|\zeta|\gg|t|$ the old cross-sections give the lower bound; for $|\zeta|=O(|t|)$ the new sections do so. The same estimate is preserved by analytic orthonormalization.

The logarithmic potentials therefore converge in $L^1_{\mathrm{loc}}$ to the limiting potential. Poincar\'e--Lelong and the homogeneity of the round-sphere determinant system give
\begin{equation}\label{eq:endcurrent}
 T_{\beta,t}\rightharpoonup[D_-]+[D_+]
       +\pi_1^*\frac{\omega_B}{2\pi}\qquad(t\to0).
\end{equation}
The two divisor multiplicities are one. In particular, this is concentration of the moving part, since the systems for $t\neq0$ have no fixed divisor.

Choose $\beta_j\downarrow0$, $\epsilon_j\downarrow0$, and $\delta_j\downarrow0$. Let $U_j=\{|x|>1-\delta_j\}$, and take a smooth cutoff $\chi_j$ supported there and equal to one on both end divisors. For each fixed $\beta_j$, choose a sufficiently small nonzero $t_j$ so that all the following hold. First, the perturbation stays positive and
\[
 \Ric(\Omega_{\beta_j,t_j})\geq(1-\epsilon_j)\Omega_{\beta_j,t_j}.
\]
This is possible because at $t=0$ the difference is at least $\epsilon_j\Omega_{\beta_j}>0$. Second, the first five eigenvalues are as close as needed to those of the product. Third, the metric and measure perturbations are small relative to the product, with relative error tending to zero. Finally, by \eqref{eq:endcurrent}, require
\begin{equation}\label{eq:endcutoff}
 \int_X\chi_jT_{\beta_j,t_j}\wedge\alpha\geq2-j^{-1},
 \qquad\alpha=\pi_1^*\frac{\omega_B}{4\pi}.
\end{equation}
All requirements are compatible since $\beta_j$ is fixed when $t_j$ is chosen.

Set
\[
 \omega_j=(1-\epsilon_j)\Omega_{\beta_j,t_j}.
\]
Ricci forms are unchanged by constant scaling, so $\Ric(\omega_j)\geq\omega_j$. The potential perturbation does not change the K\"ahler class; therefore
\begin{equation}\label{eq:perturbedvolume}
 \Vol_{\omega_j}(X)=16\pi^2(1-\epsilon_j)^2\beta_j\longrightarrow0.
\end{equation}
Continuity of the finite spectrum, \eqref{eq:surfacegap}, and the diagonal choices give $\lambda_4(\omega_j)\to1$ and $\lambda_5(\omega_j)\geq3/2$. A constant metric scaling multiplies the induced line-bundle inner product by a scalar and the gradients by a scalar. It consequently leaves the projective spectral system unchanged.

The new normalized measured limit is still $\CP^1\times I$. The relative measure control gives $\mu_j(U_j)\to0$. These are also shrinking metric end neighborhoods, since \eqref{eq:Acomparison} bounds their meridian width by $C\sqrt{\delta_j}$. On the other hand,
\[
 \int_XT_j\wedge\alpha=c_1(\OO(2,2))\cdot[\alpha]=2.
\]
Together with positivity and \eqref{eq:endcutoff}, this proves
$\int_{U_j}T_j\wedge\alpha\to2$.
Proposition~\ref{prop:basepointfree} proves the remaining assertions of Theorem~\ref{thm:counterintro}.

\subsection{The local ideal and conservation of degree}
The defect has a precise local algebraic description. For fixed $\beta$, complexify the analytic parameter locally. At an end point, one of the new sections has the form $t$ times a unit. One of the old cross-sections has the form $\zeta$ times a unit plus a multiple of $t$. All sections vanish when $t=\zeta=0$. Thus the relative base ideal is exactly
\begin{equation}\label{eq:baseideal}
 \mathfrak b=(t,\zeta).
\end{equation}
Blowing up the two centers $D_\pm\times\{0\}$ makes this ideal invertible and resolves the relative projective map; this is the universal property of the blow-up \cite[Tag 0806]{Stacks}. Each exceptional fiber over a point of an end divisor has coordinates $[\zeta:t]$. Its image is given by
\[
 \zeta A(z)+tB(z),
\]
where $A(z)$ and $B(z)$ are linearly independent by the component argument used in \eqref{eq:potentialbound}. Hence it maps with degree one to a projective line.

On a vertical curve $C=\{z\}\times\CP^1$, the original degree is $\deg\OO(2,2)|_C=2$. The reduced map on the main central component is constant on this curve. The two exceptional components carry the missing degrees:
\begin{equation}\label{eq:degreeconservation}
 2=0+1+1.
\end{equation}
Thus the integers are conserved on the modified total space, but not on the moving system of the principal limiting component alone.

A general one-dimensional version follows from the same calculation. If degree-$d$ polynomial tuples $P_j$ without common zeros converge coefficientwise to $P=QR\neq0$, with $Q$ their limiting greatest common factor and $R$ without common factor, then
\begin{equation}\label{eq:curvequantization}
 [P_j]^*\omega_{\FS}\rightharpoonup[R]^*\omega_{\FS}
             +\sum_z\operatorname{ord}_z(Q)\,\delta_z.
\end{equation}
Indeed, $\log\sum|P_{j,i}|^2$ converges locally in $L^1$ to
$\log|Q|^2+\log\sum|R_i|^2$, by subharmonic compactness and coefficient convergence. Applying $dd^c$ proves \eqref{eq:curvequantization}. Its total defect is the nonnegative integer $\deg Q$. In particular, bounded degree and exact holomorphicity do not by themselves imply absence of end defects.

\section{Logarithmic extensions and integral twisting}

\subsection{What the determinant does not retain}
The endpoint phenomenon has a second consequence. Even when the fixed divisors are retained, the highest exterior power alone need not determine the twisting of a complex filling.

\begin{proposition}\label{prop:logextension}
Let
\[
 X_p=\mathbb P\bigl(\OO\oplus\OO(-p)\bigr)
 \xrightarrow{\pi}B=\CP^{n-1},\qquad p\geq0,
\]
and let $D_-,D_+$ be its two natural disjoint sections, with normal bundles $\OO(-p)$ and $\OO(p)$. Put $D=D_-+D_+$. Then
\begin{equation}\label{eq:logarithmicdet}
 -K_{X_p}-D=\pi^*(-K_B)
\end{equation}
for every $p$. There is an exact sequence
\begin{equation}\label{eq:logextension}
 0\longrightarrow\OO_{X_p}
 \longrightarrow T_{X_p}(-\log D)
 \longrightarrow\pi^*T_B\longrightarrow0.
\end{equation}
In fiber coordinates $w_\alpha=g_{\alpha\beta}(z)w_\beta$, its extension cocycle is represented, with a consistent convention, by $\{-d\log g_{\alpha\beta}\}$. It records the line-bundle Chern class once the vertical generator $w\partial_w$ has been fixed.
\end{proposition}
\begin{proof}
The local vertical logarithmic vector fields $v_\alpha=w_\alpha\partial_{w_\alpha}$ agree on overlaps. They generate the kernel in \eqref{eq:logextension}. For a base vector field $X$, the local horizontal lifts satisfy
\begin{equation}\label{eq:logtransition}
 \widetilde X_\alpha-\widetilde X_\beta
 =-(X\log g_{\alpha\beta})v.
\end{equation}
This gives the displayed extension cocycle. Taking determinants removes the off-diagonal term, so
$\det T_{X_p}(-\log D)=\pi^*\det T_B$,
which is \eqref{eq:logarithmicdet}. The general interpretation of extensions by such Cech cocycles is recalled in \cite[Chapter V, Section 14]{Demailly}.
\end{proof}

For all $p$, the saturated anticanonical line bundle in \eqref{eq:logarithmicdet} is the same pullback from the base. Thus its determinant data do not distinguish the ruled spaces. Moreover, rescaling $v$ rescales the representative of the extension. A complex extension class without a specified primitive integral generator is not yet a twisting integer. For instance, on $\CP^1$, all nonzero classes in the one-dimensional space $\operatorname{Ext}^1(\OO(2),\OO)$ are equivalent under nonzero scalar changes of the left-hand trivialization.

With the primitive circle action $w\mapsto e^{it}w$ fixed, the period is $2\pi$ and
\[
 \frac1{2\pi i}\int_\Gamma d\log g_{\alpha\beta}
\]
recovers the integral transition degree. The normal bundle and the circle bundle then use the same transition functions, so their degrees match. A limiting real Lie algebra direction, without a closed subgroup or integral lattice, does not contain this information by itself.

In the counterexample of Section~\ref{sec:ends}, the actual fixed divisor is zero at every nonzero perturbation parameter and equals $D_-+D_+$ in the limit. Consequently the determinants of the corresponding logarithmic tangent sheaves change from $\OO(2,2)$ to $\OO(2,0)$. Their vertical degrees change from two to zero. The two exceptional end components in \eqref{eq:degreeconservation} account for this difference. If the geometric divisors $D_\pm$ are retained independently throughout this particular product family, its geometric twisting remains zero. The example concerns failure of automatic recovery from the moving spectral system, not failure of topological invariance in an already identified smooth bundle family.

\subsection{A conditional obstruction for holomorphic ruled fillings}
There is a strong obstruction once the needed holomorphic objects actually exist.
\begin{proposition}\label{prop:negativeDivisor}
Let $(M^n, \omega)$ be a compact K\"ahler manifold with $\Ric(\omega)\geq\omega$, $n\geq 2$. If $M$ contains a holomorphic divisor $D\simeq\CP^{n-1}$ with normal bundle $N_{D/M}\simeq\OO(-p)$, $p\geq0$, then
\begin{equation}\label{eq:pbound}
 p\leq n-1.
\end{equation}
If $[\omega]|_D=2\pi nA\,H$, then $0<A\leq1-p/n$.
\end{proposition}
\begin{proof}
The tangent-normal exact sequence gives $c_1(TM)|_D=(n-p)H$. On a holomorphic line $\ell\subset D$,
\[
 0<\int_\ell\omega\leq\int_\ell\Ric(\omega)=2\pi(n-p).
\]
The integrality of $p$ proves \eqref{eq:pbound}; substituting the restricted K\"ahler class gives the last assertion. Holomorphicity of the curve is essential for restricting the positive $(1,1)$-form in this way.
\end{proof}

\begin{proposition}\label{prop:ruledvolume}
On $X_p$ as in Proposition~\ref{prop:logextension}, let $p>0$ and suppose
\[
 [\omega]|_{D_-}=2\pi nA\,H,\qquad
 [\omega]|_{D_+}=2\pi nB\,H.
\]
Then every K\"ahler metric in this class satisfies
\begin{equation}\label{eq:ruledvolume}
 \Vol_\omega(X_p)=\frac{(2\pi n)^n}{n!p}(B^n-A^n),\qquad0<A<B.
\end{equation}
If $\Ric(\omega)\geq\omega$, then
\begin{equation}\label{eq:ruledoscillation}
 B-A\leq
 \left(\frac{n!(n-1)}{(2\pi n)^n}\Vol_\omega(X_p)\right)^{1/n}.
\end{equation}
For $p=0$, the two section restrictions have equal scale $A=B$.
\end{proposition}
\begin{proof}
Regard $D_\pm$ also as divisor classes. They satisfy $D_+-D_-=pH$ and $D_+D_-=0$. The K\"ahler class is
\[
 [\omega]=\frac{2\pi n}{p}(BD_+-AD_-).
\]
Its fiber integral is $2\pi n(B-A)/p$, so $B>A$. The self-intersections are
$\int D_+^n=p^{n-1}$ and $\int D_-^n=(-p)^{n-1}$.
All mixed terms vanish, giving \eqref{eq:ruledvolume}. If the Ricci inequality holds, Proposition~\ref{prop:negativeDivisor} gives $p\leq n-1$. Use also $(B-A)^n\leq B^n-A^n$ to obtain \eqref{eq:ruledoscillation}. When $p=0$, the two sections are homologous.
\end{proof}

These conclusions do not require the metric to be a Calabi-type invariant metric. A holomorphic $\CP^1$-bundle over $\CP^{n-1}$ with two disjoint holomorphic sections has precisely the indicated form: fiber coordinates sending the sections to zero and infinity have transition functions $w_\alpha=g_{\alpha\beta}w_\beta$, which define a line bundle on the base. Since $\operatorname{Pic}(\CP^{n-1})=\Z$, exchanging the sections if necessary gives $\OO(-p)$.

Consequently, any volume-collapsing family with these identified holomorphic sections has $B-A\to0$. This rules out a fixed nonzero transverse change of reduced K\"ahler class within that category. It is not a classification of all limits satisfying $\lambda_{n^2}\to1$. The full spectral action, the anticanonical projection estimate, and the model degree bounds do not by themselves supply such a holomorphic bundle, its sections, or the primitive circle lattice on every approximating manifold.

The results above thus distinguish two kinds of recovery. At multiplicity above $n^2$, the full spectrum and the kernel estimate force a noncollapsed smooth K\"ahler model. At the collapsing threshold, genuine anticanonical sections can still be recovered, but their moving systems can transfer integral degree to end components. Any further classification by complex fillings must retain those components and the logarithmic extension data rather than discard them as errors.

\end{document}